\documentclass[oneside,10pt]{amsart}
\usepackage[b5paper,scale=.81]{geometry}  
\usepackage{mathrsfs,amssymb,mathabx}
\usepackage[mathlines]{lineno}
\usepackage[dvipsnames]{xcolor}
\usepackage[pagebackref,colorlinks,citecolor=Plum,urlcolor=Periwinkle,linkcolor=DarkOrchid]{hyperref}

\usepackage{graphicx}

\theoremstyle{definition}
\newtheorem{definition}{Definition}[section]
\newtheorem{remark}[definition]{Remark}

 \theoremstyle{plain}
\newtheorem{theorem}{Theorem}

\newtheorem{lemma}[definition]{Lemma}
\newtheorem{corollary}[definition]{Corollary}
\newtheorem{proposition}[definition]{Proposition}

\title[A stable Meyer-It\^o formula]{A Meyer-It\^o Formula for Stable Processes via Fractional Calculus}
\author{Alejandro Santoyo Cano}
\author{Ger\'onimo Uribe Bravo}
\address{Instituto de Matem\'aticas\\ 
Universidad Nacional Aut\'onoma de M\'exico\\
\'Area de la Investigaci\'on Cient\'ifica, Circuito Exterior, Ciudad Universitaria\\ Coyoac\'an, 04510. Ciudad de M\'exico, M\'exico }
\newcommand{\RR}{\mathbb{R}} 
\newcommand{\NN}{\mathbb{N}} 
\newcommand{\ZZ}{\mathbb{Z}}

\newcommand{\BB}{\mathcal{B}}
\newcommand{\FF}{\mathcal{F}}
\newcommand{\GG}{\mathcal{G}}

\newcommand{\LL}{\mathcal{L}}

\newcommand{\EE}{\mathbb{E}}

\newcommand{\PP}{\mathbb{P}}

\DeclareMathOperator{\ii}{1{\hskip -2.5 pt}l}
\DeclareMathOperator{\sgn}{sgn} 
\newcommand{\stabs}{S_{\alpha}\left(  \beta, \sigma \right)}
\newcommand{\stabc}{S_{\alpha}\left(  c_-, c_+ \right)} 
\newcommand{\ClassC}{\mathcal{C}^{\alpha,c_-,c_+}}
\renewcommand{\SS}{\mathcal{S}}

\newcommand{\Phipr}{\Phi^{\prime}}

\DeclareMathOperator{\leb}{Leb}

\newcommand{\bline}{\begin{linenomath*}}
\newcommand{\eline}{\end{linenomath*}}
\thanks{GUB's research is supported by
UNAM-DGAPA-PAPIIT 
grant IN114720. 
ASC's research is supported by 
CONACyT PhD scholarship CVU 486052. 
} 
\usepackage[dvipsnames]{xcolor}

\begin{document}
\begin{abstract}
The infinitesimal generator of a one-dimensional strictly $\alpha$-stable process can be represented as a weighted sum of (right and left) Riemann-Liouville fractional derivatives of order $\alpha$ and one obtains the fractional Laplacian in the case of symmetric stable processes. 
Using this relationship, we compute the inverse of the infinitesimal generator on 
Lizorkin space, 
from which we can recover the potential if $\alpha \in (0,1)$ and the recurrent potential if $\alpha \in (1,2)$. The inverse of the infinitesimal generator is expressed in terms of a linear combination of (right and left) Riemann-Liouville fractional integrals of order $\alpha$. One can then state a class of functions that give semimartingales when applied to strictly stable processes and state a Meyer-It\^o theorem with a non-zero (occupational) local time term, providing a generalization of the Tanaka formula given by Tsukada \cite{MR3964382}. This result is used to find a Doob-Meyer (or semimartingale) decomposition for $|X_t - x|^{\gamma}$ with $X$ a recurrent strictly stable process of index $\alpha$ and  $\gamma\in (\alpha-1,\alpha)$, generalizing the work of Engelbert and Kurenok \cite{MR3943123} to the asymmetric case. 
\end{abstract}
\maketitle

\section{Introduction and statement of the results}
One might argue that the connection between fractional calculus and stable processes 
can be further strengthened 
even though links between the fractional Laplacian and symmetric stable processes are often alluded to (cf. \cite{MR3971272,MR4043885} with Remark \ref{remark:frac_lapl}). The oldest references that relate fractional calculus and stable random variables are the seminal work of Feller \cite{MR0052018}, which uses fractional calculus to compute a series for stable densities and the articles of Gorenflo and Mainardi (cf. \cite{MR1648257,GM07}), which identify a correspondence between stable characteristic functions and the Fourier transform of fractional derivatives. 
More recent references are the book of Meerschaert and Sikorskii \cite{MR3971272} and the article of Kolokoltsov \cite{MR3377407}, 
where the infinitesimal generator of a stable process and various transformations are written in terms of different types of fractional derivatives. 

One objective of this work is to present a natural application of fractional calculus to (one-dimensional and asymmetric) stable processes by inverting their infinitesimal generator. 
The inversion is valid 
in the so called Lizorkin space. 
Multiple consequences include a generalization of the celebrated Tanaka formula for Brownian motion into the stable setting as first obtained by Tsukada \cite{MR3964382}. 
This will follow from constructing a function which the  generator transforms into the $\delta$ distribution. 
More generally, one can define a class of functions whose image under the generator is a signed measure. One obtains semimartingales when applying functions of this class to stable processes; the semimartingale decomposition gives us a version of Meyer-It\^o formula for discontinuous semimartingales (cf. Protter \cite[IV.7]{MR2020294}) which features a non-zero local time term. 
This allows for a concrete semimartingale decompositions for power functions applied to stable processes which were recently obtained for symmetric stable processes in Engelbert and Kurenok \cite{MR3943123}. 

This work is based on several results, from both probability theory and fractional calculus, 
so let us first state the basic elements we will need. 

\begin{definition}[Strictly stable process]
A L\'evy process $\left(X_t\right)_{t\geq 0}$ 
is called a strictly stable process with index of stability $\alpha \in (0,2)\setminus \left\{1\right\}$  if, for any $c>0$: 
$X_{ct}\stackrel{d}{=} c^{1/\alpha}X_t$
\end{definition}

We will only consider strictly stable processes in this paper, excluding the cases when the index is $1$ or $2$, corresponding to the symmetric Cauchy process and Brownian motion, which have been studied with different techniques. Stable processes belong to the class of L\'evy processes (c.f. \cite{MR1406564}) and their properties are deferred to the next section.

According to \cite{MR1739520}(Chapter 3) there exist some constants $c_-, c_+ \geq 0$, not both zero, such that the L\'evy measure $\nu$ of $X$, which describes the jumps of $X$ and is given by
\[
\nu(A)=\mathbb{E}(\# \{t\in [0,1]: X_{t}-X_{t-}\in A\}) , 
\]satisfies:
\begin{linenomath*}
\begin{equation*} 
\nu(dh)	 = 
\left( c_{-} \ii_{\left\{ h<0 \right\}} + c_{+} \ii_{\left\{ h>0 \right\} } \right) \frac{dh}{\left|h\right|^{\alpha+1}}. 
\end{equation*}
\end{linenomath*}
Stable processes can then be constructed using a Poisson random measure $N$ with intensity $ds\, \nu(dh)$ if $\alpha\in (0,1)$ or the compensated Poisson random measure  $\tilde{N}$ when $\alpha\in (1,2)$ by means of the 
L\'evy-It\^o decomposition: 
\bline
\begin{equation*}
X_t =X_0 + \begin{cases}
\displaystyle 
 \int_0^t \int_{\RR_0} h N(ds,dh)& \text{ if }\alpha\in (0,1)
 \\ \displaystyle 
 \int_0^t \int_{\RR_0} h \tilde{N}(ds,dh)&\text{ if }\alpha\in (1,2)
\end{cases}.
\end{equation*}
\eline
In fact, note that in the recurrent case when $\alpha\in (1,2)$, $X_t$ is integrable for any $t$ and $X$ is a martingale (whenever $X_0$ is deterministic). In both cases, we will write that $X \sim \stabc$ when we refer to a strictly stable process with such parameters. 
The infinitesimal generator $\mathcal{L}$ of $X$ can be defined as the derivative at zero of the semigroup on an adequate class of functions. Indeed, recall that if $\phi:\RR\to\RR$ belongs to the Schwartz space $\mathcal{S}(\RR)$ of rapidly decreasing functions, we have
\[
	\mathcal{L}\phi(x)
	:=\left.\frac{\partial }{\partial t}\right|_{t=0} \mathbb{E}(\phi(x+X_t))
	=\begin{cases}
	\displaystyle 
	\int_{\RR_0} [\phi(x+y)-\phi(x)]\, \nu(dy)&\alpha\in(0,1)\\
	\displaystyle 
	\int_{\RR_0} [\phi(x+y)-\phi(x)-y\phi'(x)]\, \nu(dy)&\alpha\in(1,2)
	\end{cases}, 
\]as in \cite[I.2]{MR1406564}. 

The behavior and further properties of the process $X$ differ substantially whether $\alpha \in (0,1)$ or $\alpha \in (1,2)$, so they will be studied separately (as above, there are many differences in these cases, such as the transient/recurrent dichotomy, the polar/non-polar character of zero, or the bounded vs unbounded variation of the sample paths). 
Nevertheless, in both cases we get a representation of their infinitesimal generator and its inverse in terms of fractional operators.
The fractional operators we will use are the Riemann-Liouville's, these definitions and further properties can be consulted in \cite[Ch. 2]{MR1347689}. 

\begin{definition}[Riemann-Liouville fractional operators]\label{def:RLFO}
Let $\alpha \geq 0$ and $\varphi \in \mathcal{S}(\RR)$. Then, the left and right Riemann-Liouville fractional operators of order $\alpha$ applied to $\varphi$ are defined in three cases: 
\begin{itemize}
    \item For $\alpha = 0$ we get the identity operator
    \bline
    \begin{equation*}
        W_-^\alpha \varphi\left(x\right) = W_+^{\alpha} \varphi\left(x\right) := \varphi\left(x\right).
    \end{equation*}
    \eline
    \item For $\alpha > 0$, the (left and right) Riemann-Liouville fractional integrals are given by
    \bline
    \begin{eqnarray*}
    W_{-}^\alpha\varphi(x)
    &:=& \frac{1}{\Gamma\left(\alpha\right)} \int_{-\infty}^{x}\left(x-t\right)^{\alpha-1}\varphi\left(t\right)dt\quad\text{and}\\
    W_+^\alpha\varphi(x)
    &:=& \frac{1}{\Gamma\left(\alpha\right)} \int_{x}^{\infty}\left(t-x\right)^{\alpha-1}\varphi\left(t\right)dt.\\
    \end{eqnarray*}
    \eline
    \item For $n-1 < \alpha \leq n$, with $n \in \NN$, the Riemann-Liouville fractional derivatives are given by
    \bline
    \begin{eqnarray*}
    W_-^{-\alpha}\varphi(x)
    &:=& 
    \frac{d^n}{dx^n} W_-^{n-\alpha}\varphi(x) \quad\text{and} 
    \\
    W_+^{-\alpha}\varphi(x)
    &:=& 
    (-1)^n \frac{d^n}{dx^n} W_+^{n-\alpha}\varphi(x). 
    \end{eqnarray*}
    \eline
\end{itemize}
\end{definition}

\begin{remark}
\begin{enumerate}
\item In fact, both the fractional integral and derivative are defined under $L_p$ assumptions depending on $\alpha$, but we restrict to Schwartz space so that their Fourier transforms are well defined. 
\item If $\alpha > 0$, we will use the following notation for the fractional integrals and derivatives:
\begin{linenomath*}
\begin{equation*}
I^{\alpha}_\pm := W^{\alpha}_\pm\quad\text{and}\quad D^{\alpha}_\pm := W^{-\alpha}_\pm.
\end{equation*}
\end{linenomath*}
\item If $\alpha \in \NN$, then the left fractional operators $I_-^\alpha$ and $D_-^\alpha $, correspond to the iterated integral and classical differential operators of order $\alpha$. Fractional operators can be regarded as ``nice'' interpolations between their corresponding integer neighbors.
\item On an adequate domain (the so called Lizorkin space, to be introduced), 
they satisfy the group property with respect to composition:  for $\alpha,\beta\in\RR$, 
\bline
\begin{equation*}
W_-^\alpha\circ W_-^{\beta} = W_-^{\alpha+\beta} \quad\text{and}\quad W_+^{\alpha}\circ W_+^{\beta} = W_+^{\alpha+\beta}.
\end{equation*}
\eline Hence, both the left and right fractional operators of order $\alpha$ can be inverted by fractional operators of order $-\alpha$. 
\item Again on Lizorkin space, we have that $W^\beta\phi\to W^\alpha \phi$ as $\beta\to\alpha$, as follows from the expressions of the Fourier transforms in Proposition \ref{prop:FTfracop}. 
\end{enumerate}
\end{remark}
With some algebraic manipulations, the infinitesimal generator of a strictly $\alpha$-stable process can be written as a linear combination of left and right Riemann-Liouville fractional derivatives of order $\alpha$. For a detailed proof see for example the article of Kolokoltsov \cite{MR3377407} (Section 2), or the book of Meerschaert and Sikorskii \cite{MR3971272} (Section 2.2). 

\begin{proposition}[Infinitesimal generator]\label{prop: InfGen_FC}
Let $\alpha \in (0,2)\setminus \left\{1\right\}$, $c_-, c_+ \geq 0$, not both zero. If $X\sim \stabc$, then the domain of the infinitesimal generator $\LL$ of $X$ contains $\mathcal{S}(\RR)$. For $\varphi \in \mathcal{S}(\RR)$, we have: 
\bline
\begin{equation*}
\mathcal{L}\varphi\left(x\right) = M_- D_-^{\alpha}\varphi\left(x\right)  +  M_+ D_+^{\alpha}\varphi\left(x\right),
\end{equation*}
\eline
where $M_\pm =  c_\pm \Gamma(-\alpha)$. 
\end{proposition}

\begin{remark}\label{remark:frac_lapl}
This representation is consistent with the case $\alpha = 2$ and $c_- = c_+$, which corresponds to the Brownian motion, and its infinitesimal generator is the Laplacian $\Delta$. In the case $\alpha \in (0,2)\setminus \left\{1\right\}$ and $c_- = c_+$,  corresponding to a symmetric strictly $\alpha$-stable process, the infinitesimal generator is given by the fractional Laplacian $-(-\Delta)^{\alpha/2}$.
\end{remark}

The semigroup property of fractional operators is no longer enough to invert the infinitesimal generator, since we are lacking an expresion for the composition of left and right fractional operators. The result of this computation is stated in the forthcoming Proposition \ref{Proposition:frac_comp}.

The main problem working in the fractional calculus framework is the domain of definition of these operators; Schwartz space is not invariant under fractional operators (cf. \cite{MR1347689}, section 8.2). Since we are seeking for the inverse of the infinitesimal generator, it is useful to have a space which remains invariant under the action of the Riemann-Liouville fractional operators. This kind of space has been thoroughly studied by Lizorkin \cite{MR0262814, LizorkinTranslation}, Samko, Kilbas and Marichev \cite{MR1347689} and Rubin \cite{MR1428214, MR3410931}.

\begin{definition}[Lizorkin space]\label{def:Lizorkin}
Consider the space of functions that vanish at zero together with all its derivatives:
\bline
\begin{equation*}
\Psi = \left\{ \psi \in \mathcal{S}(\RR)\left|  \psi^{ (j)}(0)=0, j \in \{0,1,2,\ldots \}  \right. \right\}.
\end{equation*}
\eline
Then, the space of functions whose Fourier transforms are in $\Psi$ is called the Lizorkin space and is defined by
\bline
\begin{equation*}
\Phi = \left\{ \phi \in \mathcal{S}(\RR) \left| \FF[\phi] \in \Psi  \right. \right\}.
\end{equation*}
\eline
\end{definition}

In the Lizorkin space, compositions of fractional operators are well defined and therefore fractional integrals are the inverses of fractional derivatives. In general, to invert the generator, we need to see how crossed compositions are computed. The following result is stated, without proof, for fractional integrals in the article of Feller \cite{MR0052018}. 
\begin{proposition}\label{Proposition:frac_comp}
Let $\lambda, \mu \in \RR$ with $(\lambda + \mu) \notin \ZZ$ and $\phi\in \Phi$. 
Then,  the crossed composition of Riemann-Liouville operators satisfy:
\bline
\begin{equation}
W_+^{\lambda} W_-^{\mu} \phi \left(x\right) = \frac{\sin\left(\mu \pi\right)}{\sin\left( \left(\lambda +\mu\right) \pi \right)} W_-^{\lambda + \mu} \phi \left(x\right) +  \frac{\sin\left(\lambda \pi \right)}{\sin\left( \left(\lambda +\mu\right) \pi \right)} W_+^{\lambda + \mu} \phi \left(x\right). \label{eq:cross_comp}
\end{equation}
\eline
\end{proposition}
Working in the Lizorkin space and using the last result we can compute the inverse of the infinitesimal generator of a stable process: 
\begin{theorem}[Inverse of the Infinitesimal Generator]\label{theorem: Inverse IG}
Let $\alpha \in (0,2)\setminus \left\{1\right\}$, $c_-, c_+ \geq 0$, not both zero. 
Consider $X\sim \stabc$ with infinitesimal generator $\LL$. 
Then, $\LL$ is invertible in $\Phi$ and for every $\phi\in \Phi$
\bline
\begin{equation}
\displaystyle \mathcal{L}^{-1}\phi\left(x\right) = K_{-} I_-^{\alpha}\phi\left(x\right) + K_{+} I_+^{\alpha}\phi\left(x\right), \label{eq: Inverse IG}
\end{equation}
\eline
where
\bline
\begin{equation*}
K_\pm = \frac{M_\pm}{M_-^2 + M_+^2 + 2M_-M_+\cos(\pi \alpha)}, \quad i=1,2,
\end{equation*}
\eline
and the constants $M_i$ as defined in Proposition \ref{prop: InfGen_FC}.
\end{theorem}
Note that Lizorkin space is known to be dense in the space of continuous functions vanishing at infinity and in $L_p$  (cf. \cite{LizorkinTranslation} and \cite{MR1330207}), so that the above inversion formula is quite general. 
As an application of equation \eqref{eq: Inverse IG}, the following known results can be recovered:

\begin{enumerate}
\item For the case $\alpha \in (0,1)$, the L\'evy process $X$ is transient. 
Therefore, 
its potential corresponds to the inverse of the negative of the infinitesimal generator, $(-\LL)^{-1}$. 
The above theorem recovers the expression of Sato \cite[Example 5.4]{MR0408001}. 
\item For the case $\alpha \in (1,2)$, the L\'evy process $X$ is recurrent and its classical potential is infinite. Nevertheless, Port \cite{MR217877} defined the recurrent potential for stable processes (by an appropriate compensated kernel) and computed it explicitly. 
As Sato \cite{MR0408001} notes, for a wide class of L\'evy processes, the limit 
$\lim_{\lambda\to 0}(\lambda-\LL)^{-1}$ 
corresponds to a potential (classical or recurrent). 
On Lizorkin space, where we can explicitly compute an inversion thanks to the above theorem, 
Port's computation can be recovered. 
\item A heuristic explanation of the function involved in the Tanaka formula for strictly stable processes given by Tsukada in \cite{MR3964382} 
can be given as follows. Note that the It\^o formula for L\'evy processes (see Proposition \ref{prop:ItoFormula}) tells us that for any Schwartz function $f$, writing $g=\LL f$, we have
\[
	f(x+X_t)=f(x)+M^f_t+\int_0^t g(x+X_s)\, ds,
\]where $M^f$ is a martingale whose explicit expression is only needed later. Formally, if $g$ equals the Dirac $\delta$ distribution, the last summand equals the time that $X$ spends at $x$ on $[0,t]$, which is one guiding principle behind the construction of the local time of $X$ at $x$. Hence, if $\LL F=\delta$ (which will be given a sense in Section \ref{Prelim} and a proof in Lemma \ref{lemma: FracIntDelta}), then the local time should equal $F(x+X)-M^F$. Our formula for $\LL^{-1}$ allows us to guess a solution to $\LL F=\delta$ as a linear combination $\kappa_-(x^{-})^{\alpha-1}+\kappa_+(x^{+})^{\alpha-1}$, which is exactly the formula of Tsukada. 
That $\kappa_-\neq \kappa_+$ in general is a manifestation of the asymmetry in the jumps of $X$. 
\end{enumerate}

To state the Tanaka and Meyer-It\^o formulae for stable processes, we need more preliminaries concerning  definition of local time and an important class of admissible functions. 
We now consider $\alpha\in (1,2)$ for states to be recurrent and local time to be non trivial. 

\begin{definition}[Occupational local time]
Consider a family of random variables with two indices $\left\{ L_t^a(X): a\in \RR,t\geq 0 \right\}$. 
We will call it an occupational local time of a process $X$ if the occupation time formula is satisfied for any positive Borel measurable function $f:\RR \to [0,\infty)$:
\bline
\begin{equation*}
\int_0^t f\left(X_s\right) ds = \int_{-\infty}^{\infty} f\left(a\right) L_t^a(X) da \quad \text{a.s.}
\end{equation*}
\eline
\end{definition}

The fact that this local time exists for recurrent stable processes, as well as being jointly continuous in time and space, was established by Boylan \cite{MR158434} and Barlow \cite{MR958195}. See the textbook account in \cite[Ch. V]{MR1406564}.

The following definition corresponds to the function that appears in the Tanaka formula given by Tsukada in \cite{MR3964382}, but we will write it in our notation. 
\begin{definition}\label{definition:TanakaFunction}
For every fixed $\alpha \in (1,2)$, $c_-, c_+ \geq 0$, not both zero, we define the function $F=F^{\alpha,c_-,c_+}$ by:
\bline
\begin{equation}
F(x)=\kappa_+(x^-)^{\alpha-1}+\kappa_-(x^+)^{\alpha-1}
\label{TanakaFunction}
\end{equation}
\eline
where
\[	
	\kappa_\pm=\frac{c_\pm}{\Gamma(\alpha)\Gamma(-\alpha)[c_+^2+c_{-}^2+2c_+c_-\cos \alpha\pi]}
\]
\end{definition}
It is intentional that $\kappa_-$ accompanies $x^+$ because of Lemma \ref{lemma: FracIntDelta}. 
As we have remarked and will prove after Proposition \ref{lemma: FracIntDelta}, 
$F$ is a weak solution to the Poisson equation $\LL F= \delta$.  
If we consider adequate measures $\mu$ for which the convolution $F*\mu$ is well defined, we could regard $f(x)=(F*\mu)(x)$ as a solution to $\LL f = \mu$. 

\begin{definition}[The class $\mathcal{C}^{\alpha, c_-,c_+}$]\label{def: ClassC}
For every fixed $\alpha \in (1,2)$ and $c_-, c_+ \geq 0$, not both zero, 
the class $\mathcal{C}^{\alpha}$ is defined as 
\bline
\begin{align*}
 \displaystyle \left\{ 
f=F*\mu\left|  \mu\text{ is a 
signed measure such that } \int |x|^{\alpha-1}\, |\mu|(dx)<\infty
\right. 
\right\}. \label{TanakaClass}
\end{align*}
\eline
\end{definition}
The integrability condition on $\mu$ implies that the convolution is well defined and pointwise finite. 
The Meyer-It\^o formula will feature functions $f=F*\mu\in \mathcal{C}^{\alpha,c_-,c_+}$ where $\mu$ is finite and of compact support. 
Later, in Theorem \ref{EngKur_generalization}, 
we will consider convolutions where $\mu$ is a non compactly supported measure. 
This class of functions is quite large. Indeed, it contains the absolute value function and functions of the type $|x|^\gamma$ for $\gamma\in (\alpha-1,\alpha)$ (cf. Lemma \ref{lemma: FracIntDelta}). 
Therefore, differences of convex functions are contained in $\mathcal{C}^{\alpha, c_-,c_+}$. 
The case $\gamma=\alpha-1$ is special in that we can only prove its membership to $\mathcal{C}^{\alpha,c_-,c_+}$ in the symmetric case. 

Recall that the Meyer-It\^o theorem for semimartingales, for example from \cite{MR2020294}(Theorem 70), gives a semimartingale decomposition for $|X|$ which contains a semimartingale local time term. 
However, the latter is zero for a strictly stable process. 
For functions in the class $\mathcal{C}^{\alpha,c_-,c_+}$ we prove the following occupational Meyer-It\^o theorem, 
with a non-zero local time term. 

\begin{theorem}[Occupational Meyer-It\^o formula]\label{theorem: MeyerIto}
Let  $\alpha \in (1,2)$, $c_-, c_+ \geq 0$, not both zero, and consider a strictly stable process $X\sim \stabc$. 
Let $f=F*\mu\in \mathcal{C}^{\alpha,c_-,c_+}$ and furthermore assume that $\mu$ is finite and compactly supported. 
Then, 
\bline
\begin{equation}
    f\left(X_t\right) = f\left(X_0\right) + M_t + \int_{-\infty}^{\infty} L_t^a\left(X\right) \mu \left(da\right), \label{eq: MIformula}
\end{equation}
\eline
where
\bline
\begin{equation*}
    M_t = \int_{0}^{t}\int_{\mathbb{R}_{0}}\left[f\left(X_{s-}+h\right)-f\left(X_{s-}\right)\right]\tilde{N}\left(ds,dh\right), 
\end{equation*}
\eline
is a martingale and $L_t^a(X)$ is the occupational local time at $a$ up to time $t$ of $X$.
\end{theorem}

The novel part in this result is the representation of the semimartingale in terms of an occupational local time.
\begin{remark} 
\begin{itemize}
\item In the limiting case $\alpha=2$, we have $F_{\pm}(x)=x^{\pm}$ and the corresponding class $\mathcal{C}$ can be identified with that of differences of convex functions (cf. \cite[Thm. 6.22]{MR1121940}). 
\item For recurrent symmetric stable process, that is $\alpha \in (1,2)$ and $c_-=c_+=c >0$, we have $F(x)=\kappa_{\alpha,c}|x|^{\alpha - 1}$ for some constant $\kappa_{\alpha,c}$ (cf. \cite[Corollary 1]{MR2409011}).
\item The Tanaka formula of Tsukada \cite{MR3964382}, corresponds to the case where $f=F=F*\delta$. 
\item The compact support hypothesis of $\mu$ is sufficient to ensure the integrability of all the terms in \eqref{eq: MIformula}. Since strictly stable processes have finite $\kappa$-moments for $\kappa \in (-1,\alpha)$, for non compactly supported measures $\mu$, we would at least  need to verify (or assume) the integrability of $f(X_t)$ in $L^1(\PP)$. 
\end{itemize}
\end{remark}

In general, we cannot handle the case when $\mu$ is not compactly supported, due to the integrability restrictions of strictly stable processes. Nevertheless, in the following particular case, we obtain a 
generalization of the works of Salminen and Yor in \cite{MR2409011} and Engelbert and Kurenok \cite{MR3943123} from the symmetric to the general case. 
Formally, the result would follow from applying Theorem \ref{theorem: MeyerIto} to the infinite measure $\mu(dy)=|y|^{\gamma-\alpha}[k_-\ii_{y>0}+k_+\ii_{y<0}]\, dy$. 
Recall the definition of the constants $M_\pm$ in Proposition \ref{prop: InfGen_FC}. 
\begin{theorem}[Power decomposition]\label{EngKur_generalization}
Let $\alpha \in (1,2)$ and $c_-,c_+\geq 0$ not both zero, and consider a strictly stable process $X\sim \stabc$. Then for all $x\in \RR$ and $\gamma \in (\alpha-1,\alpha)$  we have the decomposition
\bline
\begin{eqnarray}
\left| X_t - x\right|^{\gamma} &=& \left| X_0 - x\right|^{\gamma} + \int_0^t \int_{\RR_0} \left[ \left| X_{s-} - x + h\right|^{\gamma} - \left| X_{s-} - x\right|^{\gamma}\right] \tilde{N}(ds,dh) \nonumber\\
&+&  \int_0^t \left| X_s - x\right|^{\gamma-\alpha} \left[ k_-\ii_{\{X_s>x\}}  + k_+\ii_{\{X_s<x\}} \right] ds, \label{eq:EKgen}
\end{eqnarray}
\eline
where $k_{\pm}:=k_{\pm}\left( \alpha, \gamma, c_-, c_+ \right)$ 
are given by
\bline
\begin{eqnarray*}
k_- &=& 
\frac{\Gamma(\gamma + 1)}{\Gamma(\gamma- \alpha+1)}\left[ M_+\frac{\sin\left(-\alpha \pi\right)}{\sin\left((\gamma-\alpha+1)\pi\right)} + M_- \frac{\sin\left((\gamma+1)\pi\right)}{\sin\left((\gamma-\alpha+1)\pi\right)} + M_+\right]\quad\text{and}\\
k_+ &=& \frac{\Gamma(\gamma + 1)}{\Gamma(\gamma- \alpha+1)}\left[ M_-\frac{\sin\left(-\alpha \pi\right)}{\sin\left((\gamma-\alpha+1)\pi\right)} + M_+ \frac{\sin\left((\gamma+1)\pi\right)}{\sin\left((\gamma-\alpha+1)\pi\right)} + M_-\right]. 
\end{eqnarray*}
\eline
\end{theorem}
Note that the last integral in \eqref{eq:EKgen} could be written in terms of the local time as
\bline
\begin{equation*}
    \int_{-\infty}^{\infty} \left| a - x\right|^{\gamma-\alpha} \left[ k_-\ii_{\{a>x\}}  + k_+\ii_{\{a<x\}} \right] L_t^a da.
\end{equation*}
\eline
The main result of Engelbert and Kurenok \cite{MR3943123} is that this decomposition corresponds to a submartingale, thus providing the Doob-Meyer decomposition for $\left| X_t - x\right|^{\gamma}$, when $X$ is a symmetric stable process. However, if asymmetry in the jumps of the stable process is allowed, this decomposition will not be in general a submartingale. By direct inspection, the last term of the decomposition will correspond to an increasing process if and only if $k_{\pm} \geq 0$.

The constants $k_{\pm}$ have been found and used by Fournier \cite{MR3060151} by other means and in a different context. Fournier proved pathwise uniqueness for SDEs driven by an asymmetric strictly stable process and, in order to use the Gronwall inequality, he defined a constant $\beta(a,c) \in (\alpha-1,1)$, where $a=\cos(\pi \alpha)$ and $c=c_-/c_+$, assuming $0<c_- < c_+$. Then, he proved that $k_+ = 0$ for $\gamma=\beta(a,c)$.  
We will prove in Lemma \ref{lemma: const_Fournier} that, in fact, both $k_{\pm}$ are non negative for all $\gamma \geq \beta(a,c)$ and otherwise one of them is negative. 

\begin{corollary}\label{cor:Semi-DM}
Let $a=\cos(\pi \alpha)$ and $c=(c_-\wedge c_+)/(c_-\vee c_+)$. 
Then the power decomposition in Theorem \ref{EngKur_generalization} for the process $\left| X_t - x\right|^{\gamma}$ is a submartingale if $\gamma \in [\beta(a,c),\alpha)$; whereas, for $\gamma \in (\alpha-1,\beta(a,c))$ it is a semimartingale, whose finite variation part is not monotone.
\end{corollary}

The organization of the paper is as follows. In section \ref{Prelim} we state known and preliminary results regarding strictly stable processes and fractional calculus that we need for the main results. Section \ref{MainRes} contains proofs of the main results, examining the crossed composition Proposition \ref{Proposition:frac_comp}, the Inversion Theorem \ref{theorem: Inverse IG}, the Meyer-It\^o Theorem \ref{theorem: MeyerIto} and finally the Power Decomposition Theorem \ref{EngKur_generalization}.

\section{Preliminaries: stable processes and fractional operators} \label{Prelim}
Fractional calculus has been studied almost since the invention of calculus. 
One of the most famous applications is the solution to the tautochrone problem by Abel (cf. \cite{MR3721889}). 
Even though many mathematicians have contributed to the formalization of the field; 
it was Marcel Riesz who systematized several results in terms of non-local operator theory (cf. \cite{MR33}). 
The book of Samko, Kilbas and Marichev \cite{MR1347689} will be our main reference for the theory of fractional calculus in what follows. As been pointed out in the introduction, the connection between fractional calculus and stable processes will appear naturally by means of their infinitesimal generator.

In this section we state the preliminaries, regarding stable processes and fractional operators, we will need in order to prove the results outlined in the previous section.

Following Applebaum \cite{MR2512800} (Theorem 1.2.14 and 2.4.16), we state the L\'evy-Khintchine formula and the L\'evy-It\^o decomposition for the special case of strictly stable processes.

\begin{corollary}[L\'evy-Khintchine formula]
Let $X \sim \stabc$ with $c_-, c_+ \geq 0$, not both zero. Then its characteristic exponent, $\phi(u):= \frac{1}{t}\log \EE\left( e^{iuX_t }\right)$ with $u\in \RR$, can be written as
\bline
\begin{equation*}
\phi(u) = \begin{cases}\displaystyle \int_{\RR_0} \left(e^{iuh} -1\right) \nu(dh) & \text{if $\alpha\in(0,1)$},\\
\displaystyle \int_{\RR_0} \left(e^{iuh} -1- iuh \right) \nu(dh) & \text{if $\alpha\in(1,2)$}.
\end{cases}
\end{equation*}
\eline
\end{corollary}

Moreover, it can be proved (cf. Applebaum \cite{MR2512800} Theorem 1.2.21) that in the case $\alpha \in (0,2)\setminus \left\{1\right\}$ the characteristic exponent of a stable process $X$ of index $\alpha$ is equal to:
\bline
\begin{equation}
\phi(u) = \exp\left[ -\sigma|u|^{\alpha} \left( 1 - i\beta \sgn(u) \tan\left( \frac{\pi \alpha}{2} \right) \right) \right]. \label{eq:CF_sp}
\end{equation}
\eline

Here we have another parametrization of a stable process in terms of the skewness and scale parameters, denoted by $X\sim \stabs$. We can recover the $(c_-,c_+)$ parametrization solving:
\bline
\begin{eqnarray*}
\beta &=& \frac{c_+ - c_-}{c_+ + c_-},\\
\sigma &=& -(c_+ + c_-)\Gamma(-\alpha)\cos\left( \frac{\pi \alpha}{2} \right).
\end{eqnarray*}
\eline

\begin{remark}
The characteristic exponent of a strictly stable process and the Fourier transform of the fractional operators are intrinsically related as we will see in Remark \ref{remark_FTFO}.
\end{remark}

The following result concerns the finiteness of moments for stable processes. 
For the first part, the proof can be consulted in \cite{MR3964382} and the second one in \cite{MR1406564}.
\begin{proposition}\label{prop: StableMoments}
Let $\alpha \in (0,2)\setminus \left\{1\right\}$, $c_-, c_+ \geq 0$, not both zero, and consider a strictly stable process $X\sim \stabc$. Then, the following bounds are satisfied:
\begin{enumerate}
\item  For all $t>0$, $x\in \RR$ and $0<\gamma<1$,
\bline
\begin{equation*}
    \EE \left[ |X_t - x|^{-\gamma} \right] \leq S(\alpha,\gamma) t^{-\gamma/\alpha},
\end{equation*}
\eline
where $S(\alpha,\gamma)$ is a constant which depends on $\alpha$ and $\gamma$ and is independent of $x$. 
\item For all $t>0$ and $0 \leq \gamma<\alpha$,
\bline
\begin{equation*}
    \EE \left[ |X_t|^{\gamma} \right] < \infty. 
\end{equation*}
\eline If $\gamma\geq \alpha$ and $t>0$, $X_t\not \in L_\gamma$. 
\end{enumerate}
\end{proposition}

The following proposition is a corollary of Theorem 4.4.7 in Applebaum \cite{MR2512800}, it is a  version of It\^o's formula (termed predictable in \cite{MR2409011}) for stable processes. 

\begin{proposition}[It\^o's formula]\label{prop:ItoFormula}
Let $X\sim \stabc$  with $c_-, c_+ \geq 0$, not both zero, and $\alpha \in (0,1)$ and $f\in C^2_{1+,b}$. 
Then for any $t\geq 0$, with probability 1 we have 
\bline
\begin{eqnarray*}
f\left(X_t\right) &=& f\left(X_0\right) + \int_0^t \int_{\RR_0} \left[ f\left(X_{s-}+h\right) - f\left(X_{s-}\right) \right] \tilde{N}(ds,dh)\\
&&+ \int_0^t \int_{\RR_0} \left[ f\left(X_{s}+h\right) - f\left(X_{s}\right) \right] \nu(dh)ds,
\end{eqnarray*}
\eline
and in the case $\alpha \in (1,2)$ we have
\bline
\begin{eqnarray*}
f\left(X_t\right) &=& f\left(X_0\right) + \int_0^t \int_{\RR_0} \left[ f\left(X_{s-}+h\right) - f\left(X_{s-}\right) \right] \tilde{N}(ds,dh)\\
&&+ \int_0^t \int_{\RR_0} \left[ f\left(X_{s}+h\right) - f\left(X_{s}\right) - hf^{\prime}\left(X_{s}\right) \right] \nu(dh)ds.
\end{eqnarray*}
\eline
\end{proposition}

As pointed out by Engelbert and Kurenok \cite{MR3943123} in their Remark 1.1, it is a common mistake to state the It\^o formula in terms of the infinitesimal generator $\LL$, since functions in $C^2_{1+,b}$ are not in its domain. So, when we have a function $f \in C^2_{1+,b}$, we define:
\bline
\begin{equation}
    \LL f(x) := \int_{\RR_0}  \left[ f\left(X_{s}+h\right) - f\left(X_{s}\right) - hf^{\prime}\left(X_{s}\right) \right]  \nu(dh)ds. \label{eq:Lextension}
\end{equation}
\eline
If $f \in \SS \subset C^2_{1+,b}$, then it coincides with the infinitesimal generator, 
so that $\LL$ can be considered as an extension of the infinitesimal generator in the class $C^2_{1+,b}$.

In the next proposition we rewrite the definition of fractional derivative depending on the index $\alpha$, this representation is called the generator form. The fact that they are equivalent can be found on the book of Meerschaert and Sikorskii \cite{MR3971272} and the article of Kolokoltsov \cite{MR3377407}. 

\begin{proposition}[Generator form]
Let $f\in \mathcal{S}(\RR)$ and $\alpha \in (0,2)\setminus \{1\}$. Then the generator form of the left and right fractional derivatives are as follow:
\bline
\begin{eqnarray*}
D_-^{\alpha} f\left(x\right) &=& \begin{cases}
\displaystyle \frac{1}{\Gamma\left(-\alpha\right)} \int_{0}^{\infty} \frac{f\left(x - h\right) - f\left(x\right)}{h^{1+\alpha}} dh,\quad \text{if }\alpha \in (0,1)\\
\displaystyle \frac{1}{\Gamma\left(-\alpha\right)} \int_{0}^{\infty} \frac{f\left(x - h\right) - f\left(x\right) + hf^{\prime}\left(x\right)}{h^{1+\alpha}} dh,\quad \text{if }\alpha \in (1,2)\\
\end{cases}\\
D_+^{\alpha} f\left(x\right) &=& \begin{cases}
\displaystyle \frac{1}{\Gamma\left(-\alpha\right)} \int_{0}^{\infty} \frac{f\left(x + h\right) - f\left(x\right)}{h^{1+\alpha}} dh, \quad \text{if }\alpha \in (0,1)\\
\displaystyle \frac{1}{\Gamma\left(-\alpha\right)} \int_{0}^{\infty} \frac{f\left(x + h\right) - f\left(x\right) - hf^{\prime}\left(x\right)}{h^{1+\alpha}} dh, \quad \text{if }\alpha \in (1,2)\\
\end{cases}
\end{eqnarray*}
\eline
\end{proposition}

From this generator form follows that the infinitesimal generator of strictly stable processes can be seen as a weighted sum of fractional derivatives given in Proposition \ref{prop: InfGen_FC}. Now we focus on the properties of the fractional operators that will lead us to the proof of the Inversion Theorem \ref{theorem: Inverse IG}.

The main reason to use the Lizorkin space $\Phi$, defined in the introduction \ref{def:Lizorkin}, is that the Fourier transform of fractional operators applied to functions in $\Phi$ behaves well. First, recall that for $f\in \mathcal{S}(\RR)$ the Fourier transform of $f$ is defined by:
\bline
\begin{equation*}
\FF\left[f\right](u) = \int_{\RR} f(x) e^{iux} dx.
\end{equation*}
\eline

\begin{proposition}[Fourier transform of fractional operators]\label{prop:FTfracop}
Let $f\in \Phi$ and $\alpha \geq 0$, then the Fourier transforms of the Riemann-Liouville fractional operators of index $\alpha$ satisfy the following identities:
\bline
\begin{eqnarray*}
\FF\left[D_-^{\alpha}f\right] \left(u\right) &=& \left( -iu \right)^{\alpha} \FF\left[f\right](u)\\
\FF\left[D_+^{\alpha}f\right] \left(u\right)  &=& \left( iu \right)^{\alpha} \FF\left[f\right](u)\\
\FF\left[I_-^{\alpha}f\right] \left(u\right) &=& \left( -iu \right)^{-\alpha} \FF\left[f\right](u)\\
\FF\left[I_+^{\alpha}f\right] \left(u\right)  &=& \left( iu \right)^{-\alpha} \FF\left[f\right](u).
\end{eqnarray*}
\eline
\end{proposition}

The proof of this proposition can be found in the book of Samko, Kilbas and Marichev \cite{MR1347689} Lemma 8.1. Note that one of the main features of the Lizorkin space $\Phi$, is that the Fourier transforms are well behaved near zero in such a way that the product $(\pm iu)^{-\alpha}\FF[f](u)$ is well defined. 

\begin{remark}\label{remark_FTFO}
If we take the principal branch of logarithm, we have
\bline
\begin{eqnarray*}
\left( \pm iu \right)^{\alpha} &=& |u|^{\alpha} e^{\pm i\sgn(u) \alpha \pi / 2} \\
&=& |u|^{\alpha} \left( \cos\left( \frac{\alpha \pi} {2}\right) \pm i \sgn(u) \sin\left( \frac{\alpha \pi} {2} \right)\right),
\end{eqnarray*}
\eline
for all $u, \alpha \in \RR$. These are precisely the characteristic functions  of the one sided stable processes, see equation \eqref{eq:CF_sp} with $\sigma=1$ and $\beta=\pm1$.
\end{remark}

Now we are ready to prove Proposition \ref{Proposition:frac_comp}, regarding the crossed composition of Riemann-Liouville operators. First, note that from the definition of the Riemann-Liouville operators and their Fourier transforms it is easy to verify that composition of operators of the same side, left or right, commute and satisfy the semigroup property. However, the composition of crossed operators, left with right or vice versa, is not as direct as in the previous case.
\begin{proof}[Proof of Proposition \ref{Proposition:frac_comp}] 
Since the Fourier transform characterizes a function $\phi\in\Phi$, 
we will prove that the Fourier transform of both sides of the statement coincide. 
First, using the Fourier transform of fractional operators in Proposition \ref{prop:FTfracop} 
we have:
\bline
\begin{eqnarray*}
\FF\left[W_-^{\lambda}f\right] \left(u\right) &=& \left( -iu \right)^{-\lambda} \FF\left[f\right](u)\\
&=& |u|^{-\lambda}e^{i\frac{\pi}{2}\sgn(u)\lambda}\FF\left[f\right](u), \\
\FF\left[W_+^{\mu}f\right] \left(u\right)  &=& \left( iu \right)^{-\mu} \FF\left[f\right](u)\\
&=& |u|^{-\mu}e^{-i\frac{\pi}{2}\sgn(u)\mu}\FF\left[f\right](u).
\end{eqnarray*}
\eline
Then, for the LHS of equation \eqref{eq:cross_comp} we have:
\bline
\begin{eqnarray*}
\FF\left[ W_+^{\lambda} W_-^{\mu} \phi \right]\left(u\right) &=& |u|^{-\lambda}e^{i\frac{\pi}{2}\sgn(u)\lambda} |u|^{-\mu}e^{-i\frac{\pi}{2}\sgn(u)\mu} \FF\left[\phi\right]\left( u \right)\\
&=& |u|^{-(\lambda + \mu)}e^{i\frac{\pi}{2}\sgn(u)\left( \lambda - \mu \right)} \FF\left[\phi\right]\left( u \right).
\end{eqnarray*}
\eline
On the other hand, for the RHS of equation \eqref{eq:cross_comp} we have:
\bline
\begin{eqnarray*}
&& \frac{\sin\left(\mu\pi\right)}{\sin\left( \left(\lambda +\mu\right) \pi \right)} \FF\left[W_-^{\lambda + \mu} \phi\right]\left(u\right) +  \frac{\sin\left(\lambda \pi\right)}{\sin\left( \left(\lambda +\mu\right) \pi \right)} \FF\left[W_+^{\lambda + \mu} \phi\right]\left(u\right) \\
&=&\frac{\sin\left(\mu \pi\right)}{\sin\left( \left(\lambda +\mu\right) \pi \right)}|u|^{-(\lambda + \mu)} e^{-i\frac{\pi}{2}\sgn(u)\left( \lambda + \mu \right)} \FF\left[\phi\right]\left( u \right) \\
&&+   \frac{\sin\left(\lambda \pi\right)}{\sin\left( \left(\lambda +\mu\right) \pi \right)}|u|^{-(\lambda + \mu)} e^{i\frac{\pi}{2}\sgn(u)\left( \lambda + \mu \right)} \FF\left[\phi\right]\left( u \right).
\end{eqnarray*}
\eline
In order for the LHS and the RHS to be equal, it suffices to prove that:
\bline
\begin{equation*}
e^{i\frac{\pi}{2}\sgn(u)\left( \lambda - \mu \right)} = \frac{\sin\left(\mu \pi\right)}{\sin\left( \left(\lambda +\mu\right) \pi \right)} e^{-i\frac{\pi}{2}\sgn(u)\left( \lambda + \mu \right)} +\frac{\sin\left(\lambda \pi\right)}{\sin\left( \left(\lambda +\mu\right) \pi \right)} e^{i\frac{\pi}{2}\sgn(u)\left( \lambda + \mu \right)}.
\end{equation*}
\eline

This is equivalent for the real and imaginary parts agreeing and using the formula for the sum of angles we need to prove that
\bline
\begin{eqnarray*}
\cos\left( (\lambda - \mu)\frac{\pi}{2} \right) &=& \frac{\cos\left( (\lambda + \mu) \frac{\pi}{2} \right)\left[ \sin(\mu \pi) + \sin(\lambda \pi) \right]}{\sin\left( (\lambda + \mu) \pi \right)}\\
&=& \frac{ \sin(\mu \pi) + \sin(\lambda \pi) }{2\sin\left( (\lambda + \mu) \frac{\pi}{2}\right)}, \quad \text{and}\\
\sin\left( (\lambda - \mu)\sgn(u)\frac{\pi}{2} \right) &=& \frac{\sin\left( (\lambda + \mu) \sgn(u) \frac{\pi}{2} \right)\left[ \sin(\lambda \pi) - \sin(\mu \pi) \right]}{\sin\left( (\lambda + \mu) \pi \right)}\\
&=& \frac{\sgn(u)\left[ \sin(\lambda \pi) - \sin(\mu \pi) \right]}{2\cos\left( (\lambda + \mu) \frac{\pi}{2} \right)}.
\end{eqnarray*}
\eline
These trigonometric relations are proved in lemma \ref{lemma:trig_id1}, finishing the proof.
\end{proof}

Finally, we will be interested in some distributions acting on the Lizorkin space of test functions. For the definition of the action of Riemann-Liouville operators on distributions we refer to \cite{MR1347689}(Section 8.1) and Rubin \cite{MR1428214}(section 3).

\begin{definition}\label{def:Duality}
Let $f\in\Phi^{\prime}$ and $\alpha \in \RR$. The distributions $W_-^{\alpha} f$ and $W_+^{\alpha}f$ are defined by duality:
\bline
\begin{eqnarray*}
\left(W_-^{\alpha} f, \phi \right) &=& \left( f, W_+^{\alpha} \phi \right)\\
\left( W_+^{\alpha}f, \phi \right)  &=& \left(f,W_-^{\alpha} \phi \right),
\end{eqnarray*}
\eline
for any $\phi \in \Phi$ and $(g,\phi)$ denotes the evaluation of the distribution $g$ on the function $\phi$. 
\end{definition}

Note that $\delta$ belongs to $\Phi'$ (and, indeed, any Schwartz distribution) since $\Phi$ is contained in $\SS$. 
The infinitesimal generator $\LL$ associated to $X\sim \stabc$ corresponds to a linear combination of fractional derivatives as in Proposition \ref{prop: InfGen_FC}. Then, if we denote its dual operator by $\tilde{\LL}$, this corresponds to the infinitesimal generator associated to $\tilde{X}\sim S_{\alpha}(c_+,c_-)$. So that, if we take $f\in\Phi^{\prime}$, then for any $\phi \in \Phi$ we have
\bline
\begin{equation}
    \left(\LL f, \phi \right) = \left( f, \tilde{\LL} \phi \right). \label{eq:LLduality}
\end{equation}
\eline

In the next section we are going to use these results to prove the those outlined in the Introduction.

\section{Proof of the main results}\label{MainRes}
The objective of this section is to prove the Inversion Theorem \ref{theorem: Inverse IG}, the occupational Meyer-It\^o formula stated as Theorem \ref{theorem: MeyerIto} and the Doob-Meyer/semimartingale decomposition of Theorem \ref{EngKur_generalization} together with Corollary \ref{cor:Semi-DM}. 

The representation in Proposition \ref{prop: InfGen_FC} is crucial to work out the Inversion Theorem  \ref{theorem: Inverse IG}. While it is well understood that the left (right) fractional derivatives is the inverse of the left (right) fractional integral, at least in Lizorkin space, the action of the crossed compositions, as far as we know, has not been reported yet. 
For instance, in the symmetric case, where $c_- = c_+$, the infinitesimal generator corresponds to the fractional Laplacian $-(-\Delta)^{\alpha/2}$ and its inverse operator is known in the literature as the Riesz potential (cf. \cite{MR1347689}); however, since in this case the left and right fractional derivatives merge into the fractional Laplacian, a crossed composition does not appear.

Before the proof of the Inversion Theorem \ref{theorem: Inverse IG}, we will prove one more lemma regarding the composition of fractional derivatives and integrals. Since we are taking functions in the Lizorkin space, these compositions are well defined.
\begin{lemma}[Fractional compositions]\label{lemma:frac_comp}
Let $\phi \in \Phi$ and $\alpha > 0$ with $\alpha \notin \NN$, then the compositions of fractional derivatives and integrals of order $\alpha$ satisfy:
\bline
\begin{eqnarray*}
D_-^{\alpha} I_-^{\alpha} \phi\left(x\right) &=& \phi\left(x\right),\\
D_+^{\alpha} I_+^{\alpha}\phi\left(x\right) &=& \phi\left(x\right),\\
D_-^{\alpha} I_+^{\alpha} \phi\left(x\right) + D_+^{\alpha} I_-^{\alpha} \phi\left(x\right) &=& 2\cos(\alpha \pi)\phi(x).
\end{eqnarray*}
\eline
\end{lemma}
\begin{proof}

The first two equations as well as the fact that all the compositions of fractional operators commute follow from Proposition \ref{prop:FTfracop}. For the last equation, 
we use 
Proposition \ref{Proposition:frac_comp} with $\lambda = \alpha$ and $\mu \to -\alpha$, to get the result. This last limit can be taken since the composition groups $\mu\to W^\mu_{\pm}$ are continuous on Lizorkin space. 

Using equation \eqref{eq:cross_comp} twice to obtain both cross compositions,  we have:
\bline
\begin{eqnarray}
W_-^{\lambda} W_+^{\mu}\phi \left(x\right) &+& W_+^{\lambda} W_-^{\mu}  \phi \left(x\right) \nonumber\\
&=& \frac{\sin\left(\lambda \pi\right)}{\sin\left( \left(\lambda +\mu\right) \pi \right)} W_-^{\lambda + \mu} \phi \left(x\right) +  \frac{\sin\left(\mu \pi \right)}{\sin\left( \left(\lambda +\mu\right) \pi \right)} W_+^{\lambda + \mu} \phi \left(x\right) \nonumber\\
&&+ \frac{\sin\left(\mu \pi\right)}{\sin\left( \left(\lambda +\mu\right) \pi \right)} W_-^{\lambda + \mu} \phi \left(x\right) +  \frac{\sin\left(\lambda \pi \right)}{\sin\left( \left(\lambda +\mu\right) \pi \right)} W_+^{\lambda + \mu} \phi \left(x\right) \nonumber\\
&=& \left[ \frac{\sin\left(\lambda \pi\right)}{\sin\left( \left(\lambda +\mu\right) \pi \right)} + \frac{\sin\left(\mu \pi\right)}{\sin\left( \left(\lambda +\mu\right) \pi \right)} \right] W_-^{\lambda + \mu} \phi \left(x\right) \nonumber\\
&&+ \left[ \frac{\sin\left(\mu \pi \right)}{\sin\left( \left(\lambda +\mu\right) \pi \right)}  +  \frac{\sin\left(\lambda \pi \right)}{\sin\left( \left(\lambda +\mu\right) \pi \right)} \right] W_+^{\lambda + \mu} \phi \left(x\right). \label{eq:sumcrosses}
\end{eqnarray}
\eline

Moreover, 
by l'H\^opital's rule we have:
\bline
\begin{eqnarray*}
\lim_{\mu \to -\alpha} \left( \frac{\sin\left(\mu \pi\right) +\sin\left(\alpha \pi \right)}{\sin\left( \left(\alpha +\mu\right) \pi \right)} \right) \
 =\cos(\alpha \pi).
\end{eqnarray*}
\eline
Finally, with $\lambda = \alpha$ and taking the limit $\mu \to -\alpha$ in equation \eqref{eq:sumcrosses} we have:
\bline
\begin{eqnarray*}
D_-^{\alpha} I_+^{\alpha} \phi\left(x\right) &+& D_+^{\alpha} I_-^{\alpha} \phi\left(x\right) \\
&&= \lim_{\mu \to -\alpha}  \;_{-\infty}W_x^{\alpha} \;_xW_{\infty}^{\mu}\phi\left(x\right) +  \;_xW_{\infty}^{\alpha}\;_{-\infty}W_x^{\mu} \phi\left(x\right) \\
&&= \lim_{\mu \to -\alpha} \left[ \frac{\sin\left(\alpha \pi\right)}{\sin\left( \left(\alpha +\mu\right) \pi \right)} + \frac{\sin\left(\mu \pi\right)}{\sin\left( \left(\alpha +\mu\right) \pi \right)} \right] W_-^{\alpha + \mu} \phi \left(x\right)\\
&&+ \lim_{\mu \to -\alpha} \left[ \frac{\sin\left(\mu \pi \right)}{\sin\left( \left(\alpha +\mu\right) \pi \right)}  +  \frac{\sin\left(\alpha \pi \right)}{\sin\left( \left(\alpha +\mu\right) \pi \right)} \right] W_+^{\alpha + \mu} \phi \left(x\right)\\
&&= 2\cos(\alpha \pi)\phi\left(x\right).
\end{eqnarray*}
\eline
Where we used that $W^0$ is the identity operator as in Definition \ref{def:RLFO}.
\end{proof}
We are ready to prove the invertibility of the infinitesimal generator $\LL$ in $\Phi$ 
and its expression as a weighted sum of fractional integrals of order $\alpha  \in (0,2)\setminus \left\{1\right\}$.
\begin{proof}{(Inversion Theorem \ref{theorem: Inverse IG})} 
Define the operator $\GG$ as
\bline
\begin{equation*}
    \GG \phi\left(x\right) = K_- I_-^{\alpha}\phi\left(x\right) + K_+ I_+^{\alpha}\phi\left(x\right),
\end{equation*}
\eline
we will prove that $\GG \left( \LL \phi \right) = \LL \left( \GG \phi \right) = \phi$, so that $\LL$ is invertible and $\LL^{-1} = \GG$. 
By our definition of $\GG$, we have:
\bline
\begin{eqnarray*}
\GG \left( \LL \phi \left(x\right)\right) &=& \GG \left(  M_- D_-^{\alpha}\phi\left(x\right)  +  M_+ D_+^{\alpha}\phi\left(x\right) \right) \\
&=& K_- M_- I_-^{\alpha} \left( D_-^{\alpha}\phi\left(x\right)\right)  +  K_- M_+ I_-^{\alpha} \left(  D_+^{\alpha}\phi\left(x\right) \right) \\
&+& K_+ M_- I_+^{\alpha} \left( D_-^{\alpha}\phi\left(x\right)\right)   +  K_+M_+ I_+^{\alpha} \left(D_+^{\alpha}\phi\left(x\right) \right)
\end{eqnarray*}
\eline
Substituting the values of $K_-$ and $K_+$, and defining $M = M_-^2 + M_+^2 + 2M_- M_+\cos(\alpha \pi)$ to temporarily  ease notation,  and using Lemma \ref{lemma:frac_comp}, we get
\bline
\begin{eqnarray*}
\GG \left( \LL \phi \left(x\right)\right) &=& \frac{M_-^2}{M} \phi\left(x\right) +  \frac{M_+^2}{M} \phi\left(x\right) + \frac{M_- M_+}{M} I_-^{\alpha} D_+^{\alpha}\phi\left(x\right) + \frac{M_+ M_-}{M}I_+^{\alpha} D_-^{\alpha}\phi\left(x\right) \\
&=& \frac{M_-^2 + M_+^2 + 2M_- M_+\cos(\alpha \pi)}{M} \phi\left(x\right)\\
&=& \phi\left(x\right).
\end{eqnarray*}
\eline

We conclude that $(\GG\circ \LL)\phi=\phi$ 
and analogous computations prove that $(\LL\circ \GG)\phi=\phi$. 
\end{proof}

Considering the following generalized functions in the dual space $\Phi^{\prime}$, using Definition \ref{def:Duality} 
we will prove an important relationship between the Dirac $\delta$ distribution and the power functions, which are strongly related with the strictly stable processes. Moreover, the following lemma could be regarded as the key result to obtain Tanaka type formulae. 

\begin{lemma}\label{lemma: FracIntDelta}
If $\lambda >0$, then, the (generalized) functions $f_+^{\lambda}(x):= x^{\lambda} \ii_{\{ x>0 \}}$ and $f_-^{\lambda}(x):=|x|^{\lambda} \ii_{\{ x<0 \}}$ belong to $\Phi^{\prime}$ and 
\bline
\begin{eqnarray*}
f_+^{\lambda}(x) &=& \Gamma(\lambda+1) I_-^{\lambda+1} \delta\left(x\right),\\
f_-^{\lambda}(x) &=& \Gamma(\lambda+1) I_+^{\lambda+1} \delta\left(x\right).
\end{eqnarray*}
\eline
Therefore, 
\bline
\begin{eqnarray*}
 \LL^{-1}(\delta) = F^{\alpha, c_-,c_+}. 
\end{eqnarray*}
\eline
\end{lemma}
\begin{proof}
The computation of $I^\alpha_\pm \delta$ is found in  \cite[Ch. 2\S 8,p. 153]{MR1347689}. 
It follows from Definition \ref{def:Duality}, 
the Inversion Theorem \ref{theorem: Inverse IG}
and the previous Lemma \ref{lemma: FracIntDelta} that:
\bline
\begin{eqnarray*}
 \LL^{-1}(\delta) &=& K_- I_-^{\alpha} \delta\left(x\right)  + K_+ I_+^{\alpha} \delta\left(x\right)  \\
  &=&   \frac{K_-}{\Gamma(\alpha)}f^{\alpha-1}_+(x) + \frac{K_+}{\Gamma(\alpha)}f^{\alpha-1}_-(x).
\end{eqnarray*}
\eline
If we substitute the values of $K_-$ and $K_+$ in terms of $\alpha,c_-$ and $c_+$ we will get that $\LL^{-1}(\delta) =F^{\alpha,c_-,c_+}$ in the sense of $\Phipr$ distributions. 
\end{proof}
Thus, Theorem \ref{theorem: Inverse IG} provides an insight to the function that satisfies the Tanaka formula.

The class of convolutions $f=F^{\alpha,c_-,c_+}*\mu$ in Definition \ref{def: ClassC}, is defined in such a way that the distribution induced by the measure $\mu$ coincides with $\LL f$, in the sense of $\Phipr$ distributions. As a consequence, $\mu$ can be considered as the extension of $\LL f$ from the Lizorkin space to 
the class $C_c$ of continuous functions with compact support. 
A precise version of this is contained in the following lemma. 
It is here that the \emph{completely balanced averages} of Lizorkin play a fundamental r\^ole: 
they constitute a way to approximate $\delta$ and other distributions from within Lizorkin space. 

\begin{lemma}\label{lemma: fracderdist}
Let $f \in \mathcal{C}^{\alpha,c_-,c_+}$ be given by $f=F*\mu$. 
Then, $\LL f = \mu$ in the $\Phipr$ sense;  
that is, for every $\phi\in \Phi$:
\bline
\begin{equation*}
    \left( \LL f, \phi\right) = \left( \mu, \phi\right).
\end{equation*}
\eline   
Finally, if $\mu$ is a finite measure with compact support, then 
$\phi\mapsto (\LL f,\phi)$ extends by continuity to  $\phi\mapsto (\mu,\phi)$ from $\Phi$ to $C_c$ with the topology of uniform convergence. 
\end{lemma}
\begin{proof}
Let $\phi\in\Phi$. 
Since $\alpha-1\in (0,1)$, then $x\mapsto x^{\alpha-1}$ is subadditive on $[0,\infty)$. 
Hence, 
\begin{linenomath}
\begin{align*}
	\int |f(x) \phi(x)|\, dx
	&\leq \int\int   [ |x|^{\alpha-1}+|a|^{\alpha-1}] |\phi (x)|\, |\mu|(da) \, dx
	\\&\leq | \mu|(\mathbb{R})\int |x|^{\alpha-1}|\phi (x)|\, dx + \|\phi\|_1 \int |a|^{\alpha-1} |\mu|(da)<\infty. 
\end{align*}\end{linenomath}From equation \eqref{eq:LLduality} and Fubini's theorem (justified from the previous display applied to $\tilde \LL\phi$): 
\bline
\begin{align*}
    \left( \LL f, \phi\right) 
    &=\int_{-\infty}^{\infty}  f(x) \tilde{\LL} \phi(x)dx
    = \int_{-\infty}^{\infty} \int_{-\infty}^{\infty} F^{\alpha,c_-,c_+}(x-a) \mu(da) \tilde{\LL}\phi(x)dx\\
    &= \int_{-\infty}^{\infty} \int_{-\infty}^{\infty} F^{\alpha,c_-,c_+}(x-a)  \tilde{\LL}\phi(x)dx \mu(da)
    = \int_{-\infty}^{\infty} \left( \LL^{-1}\delta_a , \tilde{\LL}\phi\right) \mu(da)\\
    &= \int_{-\infty}^{\infty} \left( \delta_a , \tilde{\LL^{-1}}\tilde{\LL}\phi\right) \mu(da)
    = \int_{-\infty}^{\infty} \left( \delta_a , \phi\right) \mu(da)
    = \int_{-\infty}^{\infty} \phi(a) \mu(da),
\end{align*}
\eline
yielding that $\LL f=\mu$ on $\Phipr$. 

Lizorkin, in \cite{LizorkinTranslation} (cf. after Definition \ref{def:Lizorkin}), gives an approximation of $\delta$ in $\Phipr$ by means of a collection of functions $\kappa_\beta\in \Phi$ with the following property. 
If $\phi\in C_c$, then $\phi_\beta:=\kappa_\beta*\phi\to \phi$ uniformly on compact sets; note that $\phi_\beta\in \Phi$.
Indeed, Lizorkin writes $\kappa_\beta=\kappa^1_\beta-\kappa^2_\beta$ where $\kappa^1_\beta$ is a centered Gaussian density of variance $2\beta^2$. 
Hence $\kappa^1_\beta*\phi\to \phi$  uniformly if $\phi\in C_c$. 
On the other hand, the proof of Theorem 1 \cite[Ch. II\S 4]{LizorkinTranslation} tells us that $\kappa^2_\beta*\phi\to 0$ uniformly on compact sets since $\phi$ is integrable. 
Hence,  $\Phi$ is dense in $C_c$
If $\mu$ is finite and of compact  support then it also has a finite moment of order $\alpha-1$ and so, by the previous paragraph, $L(F*\mu)=\mu$ in $\Phi'$. 
The bounded linear functional $\phi\mapsto (\mu,\phi)$ on $C_c$ coincides with $\phi\mapsto (\LL f, \phi)$ on $\LL$, 
so that, by denseness,  the latter extends uniquely by continuity to $C_c$. 
\end{proof}

The result in Lemma \ref{lemma: fracderdist} with the Brownian motion case, where $\LL f(x) =\frac{1}{2}\Delta f(x)$, $\mathcal{C}^{2,c,c}$ corresponds to the class of differences of convex functions, whose second derivative are signed measures.

The following results are inspired by the work of Tsukada \cite{MR3964382}, which will be generalized  by a well-known procedure to construct approximations of a function, smoothing it with mollifiers (cf. \cite{MR1121940}, Theorem 6.22), allowing us to use It\^o formula \eqref{prop:ItoFormula}.

A positive real function $\rho \in C^{\infty}_c$, with support in $[-1,1]$ and integral equal to one, is said to be a mollifier. 
Then, if we consider a sequence of functions given by $\rho_n(x) = n\rho(nx)$ for all $n\in \NN$, this sequence converges weakly to the Dirac $\delta$ distribution in the sense of Schwartz distributions, that is
\bline
\begin{equation*}
\left| \int_{-\infty}^{\infty} \rho_n(x)\phi(x)dx - \phi(0) \right| \longrightarrow 0, \quad \text{as } n\to \infty,
\end{equation*}
\eline
for all $\phi \in \SS$.

Let $C^{\infty}_{1+,b}$ be the family of continuous functions with bounded derivatives of any order greater than or equal to one. We are going to use some bounds for the function $F^{\alpha,c_-,c_+}$ as well as of its increments, for a proof of the following results we refer to \cite{MR3964382}(cf. equation (3.9) of the proof of Theorem 3.1 and the proof of Lemma 3.1 in that reference). For fixed $\alpha, c_-$ and $c_+$, to ease the notation, we are going to write $F$ instead of $F^{\alpha,c_-,c_+}$ when there is no confusion with the parameters. 
Also, recall the constants $\kappa_\pm$ in the definition of $F$ and write $\kappa=\kappa_-\vee \kappa_+$, so that $0\leq F(x)\leq \kappa|x|^{\alpha-1}$. 

\begin{lemma}\label{lemma: Tsukada_results}
Let  $\alpha \in (1,2)$, $c_-, c_+ \geq 0$, not both zero, and consider a strictly stable process $X\sim \stabc$. Consider the function $F^{\alpha,c_-,c_+}$ in equation \eqref{TanakaFunction}, then the following results are satisfied:
\begin{enumerate}
\item  Let $(\rho_n)_{n \geq 1}$ as above, then $F_n := F^{\alpha,c_-,c_+}*\rho_n \in C^{\infty}_{1+,b}$ for all $n \in \NN$ and $F_n \to F$, uniformly on compact sets as $n\to \infty$.\\
\item  Let $|h|\leq 1$, $a \in \RR$, $s> 0$ and $\epsilon_0 \leq (\alpha-1)\wedge(2-\alpha)$, then we have:
\bline
\begin{eqnarray*}
    \EE \left[\left|F(X_{s_-} -a  +h) - F(X_{s_- }- a) \right|^2 \right]&\leq& c_1  S(\alpha,2+\epsilon_0 - \alpha) s^{(\alpha-2-\epsilon_0)/\alpha} |h|^{\alpha + \epsilon_0},
\end{eqnarray*}
\eline
where $c_1 = 20\kappa^2$  and the constant $S(\cdot,\cdot)$ as in Proposition \ref{prop: StableMoments}, and the same bound holds if we replace $F$ by $F_n$. 
Moreover, this bound satisfies:
\bline
\begin{eqnarray*}
    \int_0^t \int_{|h|\leq 1} s^{(\alpha-2-\epsilon_0)/\alpha} |h|^{\alpha + \epsilon_0} \nu(dh)ds &=& \left( \frac{c_+ + c_-}{\epsilon_0} \right) \left( \frac{\alpha}{2\alpha - \epsilon_0 - 2} \right) t^{(2\alpha - \epsilon_0 - 2)/\alpha}
    < \infty.
\end{eqnarray*}
\eline
\item  Let $|h|> 1$, $a \in \RR$ and $s> 0$, then we have:
\bline
\begin{eqnarray*}
    \EE \left[ \left|F(X_s - a +h) - F(X_s -a) \right|\right] &\leq& c_2 |h|^{\alpha -1}.
\end{eqnarray*}
\eline
where $c_2 = 4\kappa$ and the same bound holds if we replace $F$ by $F_n$. 
Moreover, this bound satisfies:
\bline
\begin{equation*}
    \int_0^t \int_{|h| > 1} |h|^{\alpha -1 } \nu(dh) ds = \left( c_+ + c_- \right)t < \infty.
\end{equation*}
\eline
\end{enumerate}
\end{lemma}

The following result is a corollary of Lemma \ref{lemma: Tsukada_results} and it will be useful in several steps of the Meyer-It\^o theorem's proof.
\begin{corollary}\label{lemma: mgl_bounds}
Under the assumptions of Lemma \ref{lemma: Tsukada_results}, let $f \in \ClassC$, such that $f = F* \mu\in C^{\alpha,c_-,c_+}$ 
with $\mu$ a finite Radon measure and consider $f_n = f* \rho_n$ for $n\in \NN$. Then we have:
\bline
\begin{eqnarray*}
    &&\EE \left[ \left|f(X_{s_-} +h) - f(X_{s_-}) \right|^2\right] \leq (\mu(\RR))^2  c_1 S(\alpha,2+\epsilon_0 - \alpha) s^{(\alpha-2-\epsilon_0)/\alpha} |h|^{\alpha + \epsilon_0}, \quad |h|\leq 1,\\
    &&\EE \left[ \left|f(X_{s_-} +h) - f(X_{s_-})\right| \right] \leq \mu(\RR)c_2 |h|^{\alpha - 1} , \quad |h|> 1,
\end{eqnarray*}
\eline
and the same bounds are satisfied if we replace $f$ with $f_n$. 
These bounds are an elements of $L^1\left((0,t)\times A, \BB((0,t)\times A), \leb \otimes \nu)\right)$, with $A = [-1,1]\setminus\{0\}$ and $A = [-1,1]^c$ respectively. 
\end{corollary}
These results follow from the Lemma \ref{lemma: Tsukada_results} and an application of a Jensen-like inequality for finite measures. 

\begin{proof}[Proof of the Occupational Meyer-It\^o Formula (Theorem \ref{theorem: MeyerIto})]
Without loss of generality, we asume that $\mu$ is actually a positive measure, 
which was assumed to be finite with compact support and, therefore, with moments of order $\alpha$ and $2(\alpha-1)$. 
 Then, we have the representation:
 \bline
 \begin{eqnarray*}
 f\left(x\right) &=& \int_{-\infty}^{\infty} F^{\alpha,c_-,c_+}\left(x - a\right) \mu \left(da\right).
\end{eqnarray*}
\eline

Consider the sequences $F_n = F*\rho_n$ and $f_n = f*\rho_n = F*\rho_n*\mu$ as the infinitely differentiable approximations of $F$ and $f$ by the sequence $\{\rho_n\}_{n\geq 0}$, with $n\in \NN$, and we have that $f_n \to f$ uniformly on compact sets (\cite{MR3409135} Theorem 4.1: Properties of mollifiers).

Since $f_n \in C_{1+,b}^{\infty}\subset C_2$, using It\^o's formula (Proposition \ref{prop:ItoFormula}) we have:
\bline
\begin{equation} \label{eq: Itoaprox}
f_n(X_t) = f_n(X_0) + M_t^{n} + V_t^{n},
\end{equation}
\eline
where the last two terms are
\bline
\begin{align*}
M_t^{n} &=  \int_{0}^t \int_{\RR_0} \left[ f_n \left(X_{s-} + h \right) - f_n \left(X_{s-} \right)\right]\tilde{N}(ds,dh)\intertext{and}
V_t^{n} &= \int_{0}^t  \LL f_n (X_s) ds.
\end{align*}
\eline
Moreover, since the behavior of $M^n_t$ 
is different depending on the size of the jumps, we will consider $M^n_t =  M^{1,n}_t + M^{2,n}_t$, where
\bline
\begin{eqnarray*}
M_t^{1,n} &=&  \int_{0}^t \int_{h \leq 1} \left[ f_n\left(X_{s-} + h \right) - f_n \left(X_{s-} \right)\right]\tilde{N}(ds,dh),  \\
M_t^{2,n} &=&  \int_{0}^t \int_{h > 1} \left[ f_n\left(X_{s-} + h \right) - f_n \left(X_{s-} \right)\right]\tilde{N}(ds,dh).
\end{eqnarray*}
\eline
In a similar fashion, we define $M_t =  M_t^1 + M_t^2$, by replacing $f_n$ with $f$.

The proof consists in establishing the following steps: 
\begin{description}
\item[Step 1] $f(X_t)$ and $f_n(X_t)$ are in $L^1(\PP)$ and $f_n(X_t)\to f(X_t)$ in $L_1$. 
\item[Step 2] $M^{1}$ and $M^{1,n}$ are square integrable martingales  and $M^{1,n}_t\to M^1_t$ in $L_2$. 
\item[Step 3]  $M^{2}$ and $M^{2,n}$ are integrable martingales  and $M^{2,n}_t\to M^{2}_t$ in $L_1$. 
\item[Step 4]  $V^n_t\to \int L^a_t\, \mu(da)$ in $L_1$. 
\end{description}

Let's begin with \textbf{Step 1}. First, we provide a bound for $f(x)$ and $f_n(x)$ in terms of $x$ and which does not depend on $n$. 
Using that $\alpha-1\in (0,1)$, we have that $x\mapsto x^{\alpha-1}$ is subadditive on $[0,\infty)$, so that
\bline
\begin{eqnarray*}
0&\leq& f_n(x) = \int_{-\infty}^{\infty} f(x-y)\rho_n(y)\, dy\\
&=&  \int_{-\infty}^{\infty} \int_{-\infty}^{\infty} F(x-a-y)\rho_n(y)\, dy \,\mu(da)\\
&\leq&  \int_{-\infty}^{\infty} \int_{-1/n}^{1/n} 2\kappa (|x|^{\alpha -1} + |a|^{\alpha -1} +|y|^{\alpha -1})\rho_n(y)\, dy \, \mu(da)\\
&\leq&  2\kappa \int_{-\infty}^{\infty} (|x|^{\alpha -1} + |a|^{\alpha -1} + 1) \mu(da),
\end{eqnarray*}
\eline
which is finite for any $x\in \RR$ by the assumptions on $\mu$ and does not depend on $n$.

By similar arguments we have that 
\bline
\begin{equation}
    0\leq f(x) \leq 2\kappa \int_{-\infty}^{\infty} (|x|^{\alpha -1} + |a|^{\alpha -1} ) \mu(da). \label{eq:bound_fn}
\end{equation}
\eline
For the squared difference, using a Jensen-like inequality for finite measures, we have, 
\bline
\begin{eqnarray}
&&|f_n(x)-f(x)|^2 \leq 2|f_n(x)|^2 + 2|f(x)|^2 \nonumber \\
&\leq& 16 \kappa^2  \left(\int_{-\infty}^{\infty} (|x|^{\alpha -1} + |a|^{\alpha -1} + 1) \, \mu(da)\right)^2 \nonumber  \\
&\leq& 16 \kappa^2\mu(\RR) \int_{-\infty}^{\infty} \left((|x|^{\alpha -1} + |a|^{\alpha -1} + 1)\right)^2 \, \mu(da) \nonumber \\
&\leq& 48 \kappa^2\mu(\RR) \int_{-\infty}^{\infty} (|x|^{2\alpha -2} + |a|^{2\alpha -2} + 1) \, \mu(da)  \label{eq:flin_bound}
\end{eqnarray}
\eline
Then, similar arguments give
\bline
\begin{eqnarray*}
|f_n(X_t)|^2  &\leq&  12 \kappa^2\mu(\RR)  \int_{-\infty}^{\infty} (|X_t|^{2\alpha -2} + |a|^{2\alpha -2} + 1) \mu(da)\quad \text{and}\\
|f(X_t)|^2 &\leq&  8 \kappa^2 \mu(\RR) \int_{-\infty}^{\infty} (|X_t|^{2\alpha -2} + |a|^{2\alpha -2} ) \mu(da),\\
\end{eqnarray*}
\eline
and these bounds are independent of $n$ and belong to $L^1(\PP)$ since $0<2\alpha -2<\alpha$ and $\mu$ is a finite measure with a moment of order $2\alpha-2$. 
We can conclude that $f_n(X_t)$ and $f(X_t)$ are elements of $L^2(\PP)$. 
Moreover, 
by dominated convergence, we get
\bline
\begin{equation}
\lim_{n\to\infty}\EE\left[|f_n(X_t)-f(X_t)|^2\right]= \EE\left[\lim_{n\to\infty}|f_n(X_t)-f(X_t)|^2\right] = 0, \label{eq:L2fnf}
\end{equation}
\eline 
so that $f_n(X_t) \to f(X_t)$ in $L^2(\PP)$, which implies \textbf{Step 1}'s assertions.

Let's move to \textbf{Step 2}. In this case we are considering the jumps smaller than one, i.e. $h \leq 1$. To prove that $M^{1,n}$ is a square integrable martingale, according to Ikeda and Watanabe (\cite{MR1011252} section II.3),  we need to show that:
\bline
\begin{equation*}
    m^{1,n}_t:= \EE \left[  \int_{0}^t \int_{|h|\leq 1}  \left| f_n \left(X_{s-} + h \right) - f_n \left(X_{s-} \right) \right|^2 \nu(dh)ds  \right]  < \infty.
\end{equation*}
\eline

Since the integrand is positive and $(\mathcal{X},\BB(\mathcal{X}))$-measurable with $\mathcal{X}=(\Omega \times [-1,1]\setminus\{0\} \times [0,t])$, by the Fubini theorem (cf. \cite{MR1876169} Theorem 1.27), it suffices to prove the finiteness in any order of integration. Then, using the bound in Corollary \ref{lemma: mgl_bounds} for $|h|\leq 1$ and Lemma \ref{lemma: Tsukada_results}, we have:
\bline
\begin{eqnarray*}
    m^{1,n}_t &=&  \int_{0}^t \int_{|h|\leq 1}   \EE \left[\left| f_n \left(X_{s-} + h \right) - f_n \left(X_{s-} \right) \right|^2 \right] \nu(dh)ds \\
    &\leq&  \int_{0}^t \int_{|h|\leq 1}  (\mu(\RR))^2  c_1 S(\alpha,2+\epsilon_0 - \alpha) s^{(\alpha-2-\epsilon_0)/\alpha} |h|^{\alpha + \epsilon_0} \nu(dh)ds \\
    &\leq& (\mu(\RR))^2  c_1 S(\alpha,2+\epsilon_0 - \alpha)  \int_{0}^t \int_{|h|\leq 1}  s^{(\alpha-2-\epsilon_0)/\alpha} |h|^{\alpha + \epsilon_0} \nu(dh)ds \\
    &<& \infty.
\end{eqnarray*}
\eline
The result for $m^1_t$ follows from Corollary \ref{lemma: mgl_bounds} in a similar fashion. Hence, $M^1$ is also a square integrable martingale.  

In order to prove the convergence of $M^{1,n}_t \to M^1_t$ in $L^2(\PP)$, first note that according to Corollary \ref{lemma: mgl_bounds} we have 
\bline
\begin{eqnarray*}
\mathcal{M}^1_n &:=& \EE \left[  \left| f_n \left(X_{s-} + h \right) - f_n \left(X_{s-} \right) -  \left(f \left(X_{s-} + h \right) + f \left(X_{s-} \right) \right)\right|^2\right] \\
&\leq&  2\EE \left[  \left| f_n \left(X_{s-} + h \right) - f_n \left(X_{s-} \right) \right|^2\right] +   2\EE \left[ \left|f \left(X_{s-} + h \right) - f \left(X_{s-} \right)\right|^2\right]\\
&\leq& 4(\mu(\RR))^2  c_1 S(\alpha,2+\epsilon_0 - \alpha) s^{(\alpha-2-\epsilon_0)/\alpha} |h|^{\alpha + \epsilon_0},
\end{eqnarray*}
\eline
Thus, $(\mathcal{M}^1_n)_{n\geq 1}$ is dominated in $L^1\left((0,t)\times [-1,1]\setminus\{0\}, \BB((0,t)\times [-1,1]\setminus\{0\}), \leb \otimes \nu)\right)$.

We know that $(M^{1,n}_t-M^1_t)$ is a square integrable martingale for any $n\in \NN$, then using It\^o's isometry (\cite{MR2512800} p. 223)
and dominated convergence theorem for the sequence $(\mathcal{M}^1_n)_{n\geq 1}$ we have:
\bline
\begin{eqnarray*}
&&\lim_{n\to \infty}\EE \left[  \left| M^{1,n}_t-M^1_t \right|^2 \right]  \\
&&=\lim_{n\to \infty}  \int_{0}^t \int_{|h|\leq 1} \EE \left[  \left| f_n \left(X_{s-} + h \right) - f_n \left(X_{s-} \right) -  \left(f \left(X_{s-} + h \right) - f \left(X_{s-} \right) \right)\right|^2\right]  \nu(dh)ds  \\
&&=\int_{0}^t \int_{|h|\leq 1} \lim_{n\to \infty} \EE \left[  \left| f_n \left(X_{s-} + h \right) - f_n \left(X_{s-} \right) -  \left(f \left(X_{s-} + h \right) - f \left(X_{s-} \right) \right)\right|^2\right] \nu(dh)ds\\
&&=0.
\end{eqnarray*}
\eline
The convergence to zero of the last equation is a consequence of equation \eqref{eq:L2fnf} in \textbf{Step 1}. So that $M^{1,n}_t \to M^1_t$ in $L^2(\PP)$, ending with \textbf{Step 2}.

For \textbf{Step 3}, we are considering the jumps greater than one, i.e. $h > 1$. To prove that $M^{2,n}_t$ is a martingale, following Ikeda and Watanabe (\cite{MR1011252} section II.3) we must show:
\bline
\begin{equation*}
    m^{2,n}_t:= \EE \left[  \int_{0}^t \int_{|h|> 1}  \left| f_n \left(X_{s-} + h \right) - f_n \left(X_{s-} \right) \right| \nu(dh)ds   \right] < \infty.
\end{equation*}
\eline
Since the integrand is positive and $(\mathcal{X},\BB(\mathcal{X}))$-measurable with $\mathcal{X}=(\Omega \times [-1,1]^c \times [0,t])$, by the Fubini theorem 
it suffices to prove the finiteness in any order of integration. Then, using the bound in Corollary \ref{lemma: mgl_bounds} for $|h|> 1$ and Lemma \ref{lemma: Tsukada_results}, we have:
\bline
\begin{eqnarray*}
    m^{2,n}_t &=&  \int_{0}^t \int_{|h|> 1}   \EE \left[\left| f_n \left(X_{s-} + h \right) - f_n \left(X_{s-} \right) \right| \right] \nu(dh)ds \\
    &\leq&  \mu(\RR)\int_{0}^t \int_{|h|> 1} c_2 |h|^{\alpha - 1}   \nu(dh)ds \\
    &<& \infty.
\end{eqnarray*}
\eline
The result for $m^2_t$ follows by the same bounds in Corollary \ref{lemma: mgl_bounds}, so that $M^2_t$ is also a martingale. 
As in the previous step, to prove the convergence of $M^{2,n} \to M^2$ in $L^1(\PP)$, first note that according to Corollary \ref{lemma: mgl_bounds} we have 
\bline
\begin{eqnarray*}
\mathcal{M}^2_n &:=& \EE \left[  \left| f_n \left(X_{s-} + h \right) - f_n \left(X_{s-} \right) -  \left(f \left(X_{s-} + h \right) + f \left(X_{s-} \right) \right)\right|\right] \\
&\leq&  \EE \left[  \left| f_n \left(X_{s-} + h \right) - f_n \left(X_{s-} \right) \right|\right] +   \EE \left[ \left|f \left(X_{s-} + h \right) - f \left(X_{s-} \right)\right|\right]\\
&\leq& 2\mu(\RR) c_2 |h|^{\alpha - 1}.
\end{eqnarray*}
\eline
Thus, $(\mathcal{M}^2_n)_{n\geq 1}$ is dominated in $L^1\left((0,t)\times [-1,1]^c, \BB((0,t)\times [-1,1]^c), \leb \otimes \nu)\right)$.

We know that $(M^{2,n}_t-M^2_t)$ is a stochastic integral with respect to a Poisson random measure for any $n\in \NN$, then using Campbell's theorem (\cite{MR1207584} section 3.2)  
and dominated convergence theorem for the sequence $(\mathcal{M}^2_n)_{n\geq 1}$ we have:
\bline
\begin{eqnarray*}
&&\lim_{n\to \infty}\EE \left[  \left| M^{2,n}_t-M^2_t \right| \right] \leq   \\
&&\int_{0}^t \int_{|h| > 1} \lim_{n\to \infty} \EE \left[  \left| f_n \left(X_{s-} + h \right) - f_n \left(X_{s-} \right) -  \left(f \left(X_{s-} + h \right) - f \left(X_{s-} \right) \right)\right|\right] \nu(dh)ds.\\
&& = 0.
\end{eqnarray*}
\eline
The convergence to zero of the last equation is a consequence of equation \eqref{eq:L2fnf} in \textbf{Step 1}. So that $M^{2,n}_t \to M^2_t$ in $L^1(\PP)$, ending with \textbf{Step 3}.

By \textbf{Step 2} and \textbf{Step 3} we conclude that $M^{n}$ and $M$ in equation \eqref{eq: Itoaprox} are martingales and 
$M^n_t\to M_t$ in $L_1$. 

Finally, for \textbf{Step 4}, 
we have from equation \eqref{eq: Itoaprox} that:
\bline
\begin{eqnarray*}
 V_t^{n} &=& f_n(X_t) - f_n(X_0) - M_t^{n}\\
 &\stackrel{L^1(\PP)}{\to}& f(X_t) - f(X_0) - M_t,
\end{eqnarray*}
\eline
as $n \to \infty$, so that the limit $\lim_{n\to \infty} V_t^n(X_t) \in L^1(\PP)$. We just need to verify that this limit coincides with the one stated in the theorem.

We know that $f_n =F*(\rho_n*\mu)\in C_{1+,b}^{\infty} \cap C^{\alpha,c_-,c_+}$  is positive and measurable and that $\rho_n*\mu$ is a finite measure with compact support. 
Then $\LL f_n$ is well defined, positive and measurable as well. So, by the occupation formula we have:
\bline
\begin{equation*}
V_t^{n} = \int_{0}^t  \LL f_n (X_s) ds = \int_{-\infty}^{\infty} L_t^a \LL f_n (a) da.
\end{equation*}
\eline
Since $L_t^a(\omega) \in C_c$ for almost all $\omega \in \Omega$, Lemma \ref{lemma: fracderdist} tells us that
\bline
\begin{equation*}
V_t^{n} = \int_{-\infty}^{\infty} L_t^a \,  (\mu * \rho_n)(da),
\end{equation*}
\eline
and since $\rho_n \to \delta$ weakly as $n\to \infty$, then $(\mu * \rho_n) \to \mu$ weakly as $n\to \infty$ as well. 
Hence, 
\bline
\begin{equation*}
\left| \int_{-\infty}^{\infty} L_t^a (\mu * \rho_n)(da) -\int_{-\infty}^{\infty} L_t^a \mu(da) \right| \to 0, \quad \text{as $n\to \infty$}.
\end{equation*}
\eline

Steps 1-4 finish the proof of Theorem \ref{theorem: MeyerIto}. 
\end{proof}

For a first application, we have the Tanaka formula for asymmetric strictly stable processes.

\begin{corollary}[Tanaka formula]\label{corollary: TanakaFormula}
Let  $\alpha \in (1,2)$, $c_-, c_+ \geq 0$, not both zero, and consider a strictly stable process $X\sim \stabc$. Then, the Tanaka formula is satisfied:
\bline
\begin{equation}
F^{\alpha,c_-,c_+}\left(X_{t}-a \right) = F^{\alpha,c_-,c_+}\left(X_{0}-a\right) + M_{t}^a(X) + L_t^a(X), \label{TanakaFormula}
\end{equation}
\eline
where  $L_t^a(X)$ is the occupational local time at $a$ up to time $t$ of $X$ and  $M_t^a(X)$ is a square integrable martingale given by
\bline
\begin{equation*}
M_{t}^a(X) = \int_0^t \int_{\RR_0} \left[ F^{\alpha,c_-,c_+}\left(X_{s-}-a+h\right) - F^{\alpha,c_-,c_+}\left(X_{s-}-a\right) \right] \tilde{N}(ds,dh).
\end{equation*}
\eline
\end{corollary}
\begin{proof}
Consider the unitary measure concentrated in $a$, that is $\delta_a(E)=1$ if $a\in E$ and zero otherwise, with $f(x) = \left(F^{\alpha,c_-,c_+} * \delta_a\right)(x) = F^{\alpha,c_-,c_+}(x-a)$, using the occupational Meyer-It\^o theorem we have:
\bline
\begin{eqnarray*}
F^{\alpha,c_-,c_+}(X_t-a) &=& F^{\alpha,c_-,c_+}\left(X_0-a\right) \\ &+&\int_{0}^{t}\int_{\mathbb{R}_{0}}\left[F^{\alpha,c_-,c_+}\left(X_{s-}-a+h\right)-F^{\alpha,c_-,c_+}\left(X_{s-}-a\right)\right]\tilde{N}\left(ds,dh\right)\\
&+& \int_{-\infty}^{\infty} L_t^{x}\left(X\right) \delta_a \left(dx\right),\\
&=& F^{\alpha,c_-,c_+}\left(X_0-a\right) +  M^a_t + L_t^a\left(X\right). 
\end{eqnarray*}
\eline
\end{proof}
We turn our attention to the power decomposition of Theorem \ref{EngKur_generalization}. Our first step will be to explicitly compute the infinitesimal generator of the power functions in Lemma \ref{lemma: FracIntDelta}.=
Let $\alpha \in (1,2)$, $c_-, c_+ \geq 0$ not both zero and $\alpha -1< \gamma <\alpha$. 
From Lemma \ref{lemma: FracIntDelta} we know that $f_\pm^{\gamma}$ belong to $\Phi^{\prime}$ and can be identified with the following fractional integrals:
\bline
\begin{eqnarray*}
f_+^{\gamma}(x) &=& \Gamma(\gamma + 1)I_-^{\gamma + 1} \delta\left(x\right),\\
f_-^{\gamma}(x) &=& \Gamma(\gamma + 1)I_+^{\gamma + 1} \delta\left(x\right).
\end{eqnarray*}
\eline
Consider the infinitesimal generator evaluated at $f^\gamma_+(x)$: 
with the constants $M_\pm$ as defined in Proposition \ref{prop: InfGen_FC}, we have
\bline
\begin{eqnarray*}
\LL f_+^{\gamma}\left(x\right) &=&   M_-D_-^{\alpha}f_+^{\gamma}\left(x\right) + M_+D_+^{\alpha}f_+^{\gamma}\left(x\right) \\
&=&   \Gamma(\gamma + 1) M_-D_-^{\alpha} I_-^{\gamma + 1} \delta\left(x\right)+  \Gamma(\gamma + 1) M_+D_+^{\alpha} I_-^{\gamma + 1} \delta\left(x\right)
\end{eqnarray*}
\eline
Using the fractional composition formulas in Lemma \ref{lemma:frac_comp}, we get
\bline
\begin{eqnarray*}
\LL f_+^{\gamma}\left(x\right) &=& \Gamma(\gamma + 1) M_- I_-^{\gamma- \alpha + 1} \delta\left(x\right) +  \Gamma(\gamma + 1) M_+ \frac{\sin\left((\gamma+1)\pi\right)}{\sin\left((\gamma-\alpha+1)\pi\right)} I_-^{\gamma- \alpha + 1} \delta\left(x\right) \\
&&+ \Gamma(\gamma + 1) M_+  \frac{\sin\left(-\alpha \pi\right)}{\sin\left((\gamma-\alpha+1)\pi\right)} I_+^{\gamma- \alpha + 1}\delta\left(x\right)    \\
&=&  \frac{\Gamma(\gamma + 1) M_-}{\Gamma(\gamma- \alpha+1)}f_+^{\gamma- \alpha}(x) +  \frac{\Gamma(\gamma + 1) M_+}{\Gamma(\gamma- \alpha+1)} \frac{\sin\left((\gamma+1)\pi\right)}{\sin\left((\gamma-\alpha+1)\pi\right)}  f_+^{\gamma- \alpha}(x)\\
&&+ \frac{\Gamma(\gamma + 1) M_+ }{\Gamma(\gamma- \alpha+1)}  \frac{\sin\left(-\alpha \pi\right)}{\sin\left((\gamma-\alpha+1)\pi\right)} f_-^{\gamma- \alpha}(x).
\end{eqnarray*}
\eline

For the function $f_-^{\gamma}(x)$, we can proceed similarly
to get
\bline
\begin{eqnarray*}
\LL f_-^{\gamma}\left(x\right) &=& \Gamma(\gamma + 1) M_-  \frac{\sin\left(-\alpha \pi\right)}{\sin\left((\gamma-\alpha+1)\pi\right)} I_-^{\gamma- \alpha + 1} \delta\left(x\right)\\
&&+ \Gamma(\gamma + 1) M_- \frac{\sin\left((\gamma+1)\pi\right)}{\sin\left((\gamma-\alpha+1)\pi\right)} I_+^{\gamma- \alpha + 1}\delta\left(x\right)  +  \Gamma(\gamma + 1) M_+ I_+^{\gamma- \alpha + 1} \delta\left(x\right)\\
&=&  \frac{\Gamma(\gamma + 1) M_-}{\Gamma(\gamma- \alpha+1)} \frac{\sin\left(-\alpha \pi\right)}{\sin\left((\gamma-\alpha+1)\pi\right)}  f_+^{\gamma- \alpha}(x)\\
&&+ \frac{\Gamma(\gamma + 1) M_- }{\Gamma(\gamma- \alpha+1)}  \frac{\sin\left((\gamma+1)\pi\right)}{\sin\left((\gamma-\alpha+1)\pi\right)} f_-^{\gamma- \alpha}(x) + \frac{\Gamma(\gamma + 1) M_+}{\Gamma(\gamma- \alpha+1)} f_-^{\gamma- \alpha}(x).
\end{eqnarray*}
\eline
Before we prove Theorem \ref{EngKur_generalization}, we need to undestand the constants $k_{\pm}\left( \alpha, \gamma, c_-, c_+ \right)$ that are used there. They play an important role in the bounded variation part of  the power decomposition \eqref{eq:EKgen}, because in order to be an increasing process, both need to be positive. The following lemma states the critical exponent $\gamma$ from which both $k_{\pm}\left( \alpha, \gamma, c_-, c_+ \right)$ are positive. Recall the definition of $c$ in Corollary \ref{cor:Semi-DM}. 

\begin{lemma}\label{lemma: const_Fournier}
Let $\alpha \in (1,2)$, $\gamma \in (\alpha-1,\alpha)$ and $k_{\pm}\left( \alpha, \gamma, c_-, c_+ \right)$ as in Theorem \ref{EngKur_generalization}. 
Define
\bline
\begin{equation*}
    \beta(a,c) := \frac{1}{\pi} \arccos\left( \frac{c^2(1-a^2)-(1+ac)^2}{c^2(1-a^2)+(1+ac)^2} \right) \in (\alpha-1,1),
\end{equation*}
\eline
where $a=\cos(\alpha \pi)$ and $c=\frac{\min(c_-,c_+)}{\max(c_-,c_+)}$. 
Then, if $c_-<c_+$ we have that $k_-\left( \alpha, \gamma, c_-, c_+ \right)$ is positive for all $\gamma \in (\alpha-1,\alpha)$ while $k_+\left( \alpha, \gamma, c_-, c_+ \right)$ is negative if $\gamma \in (\alpha-1, \beta(a,c))$ and positive if $\gamma \in ( \beta(a,c),1)$. The same conclusion follows for $c_+<c_-$ after switching the roles of the $k_{\pm}\left( \alpha, \gamma, c_-, c_+ \right)$.
\end{lemma}
\begin{proof}
Assume that $c_- < c_+$.
First, we prove $k_-\left( \alpha, \gamma, c_-, c_+ \right)>0$ for all $\gamma \in (\alpha-1,\alpha)$. Note that:
\bline
\begin{eqnarray*}
&& k_-\left( \alpha, \gamma, c_-, c_+ \right) = 
\frac{\Gamma(\gamma + 1)}{\Gamma(\gamma- \alpha+1)}\left[ M_+\frac{\sin\left(-\alpha \pi\right)}{\sin\left((\gamma-\alpha+1)\pi\right)} + M_- \frac{\sin\left((\gamma+1)\pi\right)}{\sin\left((\gamma-\alpha+1)\pi\right)} + M_+\right]\\
&&= 
\frac{\Gamma(\gamma + 1)M_+}{\Gamma(\gamma- \alpha+1)\sin\left((\gamma-\alpha+1)\pi\right)}\left[\sin\left(-\alpha \pi\right) + c \sin\left((\gamma+1)\pi\right) + \sin\left((\gamma-\alpha+1)\pi\right)\right]\\
&&= 
\frac{\Gamma(\gamma + 1)c_+\Gamma(-\alpha)}{\Gamma(\gamma- \alpha+1)\sin\left((\gamma-\alpha+1)\pi\right)}\left[\sin\left(-\alpha \pi\right) - c \sin\left(\gamma\pi\right) - \sin\left((\gamma-\alpha)\pi\right)\right].
\end{eqnarray*}
\eline
Since we have
\bline
\begin{equation*}
    \frac{\Gamma(\gamma + 1)c_+\Gamma(-\alpha)}{\Gamma(\gamma- \alpha+1)\sin\left((\gamma-\alpha+1)\pi\right)} > 0,
\end{equation*}
\eline
for all $\alpha \in (1,2)$ and $\gamma \in (\alpha-1,\alpha)$, then $k_-\left( \alpha, \gamma, c_-, c_+ \right)>0$ is equivalent to:
\bline
\begin{equation*}
    h_-(\gamma):=\sin\left(-\alpha \pi\right) - c \sin\left(\gamma\pi\right) - \sin\left((\gamma-\alpha)\pi\right) > 0,
\end{equation*}
\eline
for all $\gamma \in (\alpha-1,\alpha)$. 
Lemma \ref{Lemma_periodicitySumSines} tells us that $h_\pm$ are $2$-periodic. 
Moreover, we have that:
\bline
\begin{eqnarray*}
h_-(0) &=&\sin\left(-\alpha \pi\right) - \sin\left(-\alpha\pi\right) = 0,\\
h_-(\alpha-1) &=& \sin\left(-\alpha \pi\right) - c \sin\left((\alpha-1)\pi\right)\\
&=& \sin\left(-\alpha \pi\right)(1-c)\\
&>&0,
\end{eqnarray*}
\eline
because $c<1$ and $\alpha \in (1,2)$. This means that $h_-(\gamma)$ has just one zero in $(0,2)$  and it is before $\alpha-1$, so that $h(\gamma)>0$ for all $\gamma \in (\alpha-1,\alpha)$, as well as $k_-\left( \alpha, \gamma, c_-, c_+ \right)>0$ in the same interval. 

We will prove in a similar way the change of signs of $k_+\left( \alpha, \gamma, c_-, c_+ \right)$. Note that, as in the previous case, we just need to analyze the change of signs of the function:
\bline
\begin{equation*}
    h_+(\gamma):=c\sin\left(-\alpha \pi\right) - \sin\left(\gamma\pi\right) - c\sin\left((\gamma-\alpha)\pi\right).
\end{equation*}
\eline
Since \bline
\begin{equation*}
    h_+(0):=c\sin\left(-\alpha \pi\right) - c\sin\left(-\alpha\pi\right) = 0,
\end{equation*}
\eline
there must be just one zero in $(0,2\pi)$, this zero is precisely $\gamma = \beta(a,c)$ (\cite{MR3060151} Lemma 9). But, by definition $\beta(a,c) \in (\alpha-1,1)$, this means that:
\bline
\begin{eqnarray*}
    h_+(\gamma) &<& 0,\quad \text{if $\gamma \in (\alpha-1,\beta(a,c))$ and}\\
    h_+(\gamma) &\geq& 0,\quad \text{if $\gamma \in [\beta(a,c),1)$}.
\end{eqnarray*}
\eline
Finally, when $c_+ < c_-$, just note that since $k_+\left( \alpha, \gamma, c_-, c_+ \right) = k_-\left( \alpha, \gamma, c_+, c_- \right)$ we can use the same proof.
\end{proof}

We are ready to prove the power decomposition theorem. These results are a generalization of the works of Salminen and Yor \cite{MR2409011} and of Engelbert and Kurenok \cite{MR3943123}. The proof of the decomposition uses the Tanaka formula for asymmetric stable processes \eqref{TanakaFormula} and relies on the representation of the infinitesimal generator of a power function given in Lemma \ref{lemma: FracIntDelta}. Note that in \cite{MR2409011} it was easy to find the measure which could recover the power decomposition in the symmetric case and for the generalization we made direct use of fractional calculus to find the relevant measure needed for the asymmetric case.

\begin{proof}( of Theorem \ref{EngKur_generalization})
Recall that from Lemma \ref{lemma: FracIntDelta} we know that for $f(y)=|y|^{\gamma}$, the infinitesimal generator associated to $f$ is given by:
\bline
\begin{equation*}
    \mu(dy) = \left(k_-  \left|y\right|^{\gamma-\alpha}  \ii_{\{y>0\}} + k_+\left| y\right|^{\gamma-\alpha}  \ii_{\{0<y\}}\right)dy.
\end{equation*}
\eline
Taking the Tanaka formula \eqref{TanakaFormula} at the level $a$ and integrating both sides by $\mu^x(da)$ (the measure $\mu$ translated by $x$) we have:
\begin{equation*}
\int_{\infty}^{\infty}F\left(X_{t}-a \right) \mu^x(da) = \int_{\infty}^{\infty}F\left(X_{0}-a\right)\mu^x(da) + \int_{\infty}^{\infty}M_{t}^a(X) \mu^x(da)+ \int_{\infty}^{\infty}L_t^a(X)\mu^x(da).
\end{equation*}

Note that the representation of $f$ as a member of the Class $\ClassC$ is precisely $F*\mu$. 
We will now use a version of Fubini's theorem  for compensated Poisson random measures and apply it to the small jumps of $M^a(X)$ above. 
See \cite[Lemma A.1.2]{MR3243582}. 
We need to verify some integrability assumptions to apply it, which are 
\eqref{eq_FubiniEK} and \eqref{eq_FubiniEK2} below. 
Applying the Fubini theorem, we get
\bline
\begin{eqnarray*}
\left| X_t - x\right|^{\gamma} &=& \left| X_0 - x\right|^{\gamma} + \int_0^t \int_{\RR_0} \left[ \left| X_{s-} - x + h\right|^{\gamma} - \left| X_{s-} - x\right|^{\gamma}\right] \tilde{N}(ds,dh) \nonumber\\
&+&  \int_{-\infty}^{\infty} \left| a - x\right|^{\gamma-\alpha} \left[ k_-\ii_{\{a>x\}}  + k_+\ii_{\{a<x\}} \right] L_t^a da.
\end{eqnarray*}
\eline

Using the occupational formula for the local time, the last integral is equivalent to
\bline
\begin{equation*}
    \int_0^t \left| X_s - x\right|^{\gamma-\alpha} \left[ k_-\ii_{\{X_s>x\}}  + k_+\ii_{\{X_s<x\}} \right] ds.
\end{equation*}
\eline
This finishes the proof modulo showing that the first integral is a martingale and the applicability of Fubini's theorem. 
The proof of the martingale character will follow the ideas of \cite[Section 3]{MR3943123}. 
Incidentally, the same argument will justify the application of Fubini's theorem above. 
We can identify two cases depending on the size of the jump:
\bline
\begin{eqnarray*}
M_t^{\gamma} &=&\int_0^{t} \int_{\RR_0} \left[ \left| X_{s-} - x + h\right|^{\gamma} - \left| X_{s-} - x\right|^{\gamma}\right] \tilde{N}(ds,dh) \nonumber \\
&=& M^{\gamma,1}_t + M^{\gamma,2}_t\\
&:=&\int_0^{t} \int_{|h|\leq |X_{s-} - x |} \left[ \left| X_{s-} - x + h\right|^{\gamma} - \left| X_{s-} - x\right|^{\gamma}\right] \tilde{N}(ds,dh)\\
&&+\int_0^{t} \int_{|h|> |X_{s-} - x |} \left[ \left| X_{s-} - x + h\right|^{\gamma} - \left| X_{s-} - x\right|^{\gamma}\right] \tilde{N}(ds,dh). 
\end{eqnarray*}
\eline

In order to prove that $M^{1,\gamma}$ is a square integrable martingale, according to Ikeda and Watanabe (\cite{MR1011252} section II.3) we need to show that:
\bline
\begin{equation}
    m^{1,\gamma}_t:= \EE \left[  \int_{0}^t \int_{|h|\leq |X_{s-} - x |}  \left| \left| X_{s-} - x + h\right|^{\gamma} - \left| X_{s-} - x\right|^{\gamma} \right|^2 \nu(dh)ds  \right]  < \infty.
    \label{eq_FubiniEK}
\end{equation}
\eline
Take $\overline c=c_- \vee c_+$, then the intensity measure $\nu_{\overline c}(dh)=\overline c|h|^{-\alpha-1}dh$ is greater than the intensity measure $\nu(dh)$, corresponding to $X_t$, and if we consider the change of variable $h=(X_{s-} - x)u$ we have:
\bline
\begin{eqnarray*}
    m^{1,\gamma}_t &\leq& \EE \left[  \int_{0}^t \int_{|(X_{s-} - x)u|\leq |X_{s-} - x |} \frac{\overline c \left| \left| X_{s-} - x + (X_{s-} - x)u\right|^{\gamma} - \left| X_{s-} - x\right|^{\gamma} \right|^2 }{|X_{s-}-x|^{\alpha}|u|^{\alpha+1}} \quad du ds  \right] \\
    &=& \EE \left[  \int_{0}^t \int_{|u|\leq 1} |X_{s-}-x|^{2\gamma} \left( \left|1+u\right|^{\gamma} - 1 \right)^2  \frac{\overline c}{|X_{s-}-x|^{\alpha}|u|^{\alpha+1}}du ds  \right] \\
    &=& \EE \left[  \int_{0}^t |X_{s-}-x|^{2\gamma-\alpha}ds \right] \int_{|u|\leq 1} \left( \left|1+u\right|^{\gamma} - 1 \right)^2  \frac{\overline c}{|u|^{\alpha+1}}du.
\end{eqnarray*}
\eline
Since $-1<\alpha-2<2\gamma-\alpha<\alpha$, the integral $\displaystyle \EE \left[  \int_{0}^t |X_{s-}-x|^{2\gamma-\alpha}ds \right]$ is finite for all $t\geq 0$. It remains to check that the second integral is finite. Consider the auxiliary function $g(u) = |1+u|^{\gamma}$, and note that for any $u\in (-1,1)$ we have that $g(u)=(1+u)^{\gamma}$, which is differentiable. By the mean value theorem we can choose $u_*\in(-1,0)$ and $u^*\in(0,1)$ such that:
\bline
\begin{equation*}
    f(u)-f(0) = \begin{cases}
    f^{\prime}(u_*)u & -1<u<0,\\
    f^{\prime}(u^*)u & 0<u<1,
    \end{cases}
\end{equation*}
\eline
This corresponds to:
\bline
\begin{equation*}
    (1+u)^{\gamma}-1 = \begin{cases}
    \gamma (1+u_*)^{\gamma-1}u & -1<u<0,\\
    \gamma (1+u^*)^{\gamma-1}u & 0<u<1.
    \end{cases}
\end{equation*}
\eline
We get the following bound for any $u\in (-1,1)$:
\bline
\begin{equation*}
    |(1+u)^{\gamma}-1| \leq \gamma c_1(\gamma)|u|,
\end{equation*}
\eline
where $c_1(\gamma) = \max((1+u_*)^{\gamma-1},(1+u^*)^{\gamma-1})$. Then, we have that
\bline
\begin{eqnarray*}
\int_{|u|\leq 1} \left( \left|1+u\right|^{\gamma} - 1 \right)^2  \frac{\overline c}{|u|^{\alpha+1}}du &\leq& \gamma^2 c_1^2(\gamma) \int_{|u|\leq 1}  \overline c |u|^{1-\alpha} du\\
&\leq& \gamma^2 c_1^2(\gamma) \frac{2\overline c}{2-\alpha}\\
&<& \infty.
\end{eqnarray*}
\eline
So that $m^{1,\gamma}_t$ for any $t\geq 0$ and $M^{1,\gamma}$ is a square integrable martingale.

Now, to prove that $M^{2,\gamma}$ is a martingale, according to Ikeda and Watanabe (\cite{MR1011252} section II.3) we need to show that:
\bline
\begin{equation}
    m^{2,\gamma}_t:= \EE \left[  \int_{0}^t \int_{|h|> |X_{s-} - x |}  \left| \left| X_{s-} - x + h\right|^{\gamma} - \left| X_{s-} - x\right|^{\gamma} \right| \nu(dh)ds  \right]  < \infty.
    \label{eq_FubiniEK2}
\end{equation}
\eline
Similarly,  we have:
\bline
\begin{eqnarray*}
    m^{2,\gamma}_t &\leq& \EE \left[  \int_{0}^t \int_{|(X_{s-} - x)u|> |X_{s-} - x |} \frac{c \left| \left| X_{s-} - x + (X_{s-} - x)u\right|^{\gamma} - \left| X_{s-} - x\right|^{\gamma} \right|}{|X_{s-}-x|^{\alpha}|u|^{\alpha+1}} \quad du ds  \right] \\
    &=& \EE \left[  \int_{0}^t \int_{|u|> 1} |X_{s-}-x|^{\gamma} \left| \left|1+u\right|^{\gamma} - 1 \right|  \frac{c}{|X_{s-}-x|^{\alpha}|u|^{\alpha+1}}du ds  \right] \\
    &=& \EE \left[  \int_{0}^t |X_{s-}-x|^{\gamma-\alpha}ds \right] \int_{|u|> 1} \left| \left|1+u\right|^{\gamma} - 1 \right|  \frac{c}{|u|^{\alpha+1}}du\\
    &<&\infty,
\end{eqnarray*}
\eline
since $\gamma-\alpha \in (-1,0)$ and this moment of $X_t$ is finite for any $t\geq 0$, the expectation is finite. To see the last integral is finite, just note that $\left| \left|1+u\right|^{\gamma} - 1 \right| $ behaves like $|u|^{\gamma}$ as $|u|\to \infty$. Then, we have that $m^{2,\gamma}_t$ is finite for any $t\geq 0$ and we can conclude that $M^{2,\gamma}$ is a martingale. This allow us to conclude that $M^{\gamma} = M^{1,\gamma}+M^{2,\gamma}$ is a martingale.
\end{proof}

Finally, we state when this power decomposition is a submartingale or just a semimartingale.
\begin{proof}(of Corollary \ref{cor:Semi-DM})
From the Lemma \ref{lemma: const_Fournier} and Theorem \ref{EngKur_generalization} the last integral is a non decreasing process if and only if $\gamma \in [\beta(a,c),\alpha)$, by Lemma \ref{lemma: const_Fournier}, so that we get a Doob-Meyer decomposition for $|X_t -x|^{\gamma}$. In the other case, $\gamma \in (\alpha-1,\beta(a,c))$, this results in a semimartingale instead of a submartingale.
\end{proof}

\appendix
\section{Trigonometric results}\label{App:trig_results}
The following trigonometric result is used in the proof of the composition of crossed fractional operators.
\begin{lemma}\label{lemma:trig_id1}
Let $\lambda, \mu \in \RR$, then the following identity holds
\bline
\begin{equation*}
\cos \left( \left(\lambda - \mu \right) \frac{\pi}{2}\right) \sin \left( \left(\lambda + \mu \right) \frac{\pi}{2}\right) =  \sin \left( \mu \frac{\pi}{2} \right) \cos \left( \mu \frac{\pi}{2} \right)  + \sin \left( \lambda \frac{\pi}{2} \right)  \cos \left( \lambda \frac{\pi}{2} \right).
\end{equation*}
\eline
\end{lemma}
\begin{proof}
Using the trigonometric identities for the sum of angles we start from the LHS:
\bline
\begin{eqnarray*}
&&\left[ \cos \left( \lambda \frac{\pi}{2} \right) \cos \left( \mu \frac{\pi}{2} \right)+ \sin \left( \lambda \frac{\pi}{2} \right) \sin \left( \mu \frac{\pi}{2} \right) \right] \left[ \sin \left( \lambda \frac{\pi}{2} \right) \cos \left( \mu \frac{\pi}{2} \right)+ \cos \left( \lambda \frac{\pi}{2} \right) \sin \left( \mu \frac{\pi}{2} \right) \right]\\
&&= \cos \left( \lambda \frac{\pi}{2} \right) \sin \left( \lambda \frac{\pi}{2} \right) \cos^2 \left( \mu \frac{\pi}{2} \right) + \cos \left( \mu \frac{\pi}{2} \right) \sin \left( \mu \frac{\pi}{2} \right) \cos^2 \left( \lambda \frac{\pi}{2} \right)\\
&&+ \sin \left( \mu \frac{\pi}{2} \right) \cos \left( \mu \frac{\pi}{2} \right) \sin^2 \left( \lambda \frac{\pi}{2} \right) + \sin \left( \lambda \frac{\pi}{2} \right) \cos \left( \lambda \frac{\pi}{2} \right) \sin^2 \left( \mu \frac{\pi}{2} \right)\\
&&= \sin \left( \mu \frac{\pi}{2} \right) \cos \left( \mu \frac{\pi}{2} \right) \left[ \cos^2 \left( \lambda \frac{\pi}{2} \right) + \sin^2 \left( \lambda \frac{\pi}{2} \right) \right] \\
&&+ \sin \left( \lambda \frac{\pi}{2} \right) \cos \left( \lambda \frac{\pi}{2} \right) \left[ \cos^2 \left( \mu \frac{\pi}{2} \right) + \sin^2 \left( \mu \frac{\pi}{2} \right) \right] \\
&& = \sin \left( \mu \frac{\pi}{2} \right) \cos \left( \mu \frac{\pi}{2} \right) + \sin \left( \lambda \frac{\pi}{2} \right) \cos \left( \lambda \frac{\pi}{2} \right).
\end{eqnarray*}
\eline
\end{proof}
The following lemma is used to analyze the constant $\beta(\alpha, c)$ in Theorem \ref{cor:Semi-DM}. 
\begin{lemma}
The functions $h_{\pm}$ of Lemma \ref{lemma: const_Fournier} have minimum period $2$. 
\label{Lemma_periodicitySumSines}
\end{lemma}
\begin{proof}
Let $f_{\pm}(x)=h_{\pm}(x/2\pi)$, so that we now wish to prove that the minimum period of $f_\pm$ is $2\pi$. 
First, note that $f_{\pm}$ is a solution to $f''+f=0$. 
Second, all solutions to the above ODE are given by $a\cos+b\sin$. 
Finally, we assert that the minimum period of the above linear combination is $2\pi$ as long as  $a$ and $b$ are not both zero. 
Let us assume that $a\neq0$. 
If $\tilde f_{\pm}(x)=f_\pm(x+p)$ for some $p$, by equating initial conditions at zero, we obtain
\[
	a=a\cos p+b\sin p\quad\text{and}\quad b=-a\sin p+b\cos p	. 
\]By substituting the value for $b$ obtained in the second equation in the first and cancelling $a$, since it is non-zero, we get
\[
	1-\cos^2p=\sin^2(p)=(1-\cos p)^2. 
\]Expanding the square, we get
\[
	\cos p=\cos^2 p
\]from which $p=2k\pi$. 
The case when $b\neq 0$ is handled similarly. 
\end{proof}
\bibliography{general_math.bib}
\bibliographystyle{amsalpha}
\end{document}